\documentclass[11pt]{amsart}
\usepackage{amsfonts, amssymb, amsmath, amsthm}
\usepackage{amsfonts}
\usepackage[dvips]{graphics}
\usepackage{epsfig}
\usepackage{graphicx}
\usepackage{psfrag}
\usepackage{amsmath}
\usepackage{amssymb,subfigure}
\usepackage[numbers]{natbib}

\usepackage{textcomp}
\usepackage{float}
\usepackage{color}

\numberwithin{equation}{section} \oddsidemargin=-.0cm
\usepackage{pifont}

\usepackage{footmisc}
\usepackage[usenames,dvipsnames]{xcolor}
\usepackage{multirow}

\usepackage{amsfonts, amssymb, amsmath, amsthm}

\usepackage[colorlinks=true, pdfstartview=FitV, linkcolor=blue,
            citecolor=blue, urlcolor=blue]{hyperref}  

\def\beas{\begin{equation} \begin{aligned}}
\def\eeas{\end{aligned} \end{equation}}

\def\be{\begin{equation}}
\def\ee{ \end{equation}}

\def\bea{\begin{equation} \begin{aligned}}
\def\eea{\end{aligned} \end{equation}}

\newtheorem{rem}{Remark}
\def\br{\begin{rem}}
\def\er{\end{rem}}

\newtheorem{thm}{Theorem}[section]
\newtheorem{prop}[thm]{Proposition}
\newtheorem{corr}[thm]{Corollary}
\newtheorem{lem}{Lemma}[section]

\def\d{\mathrm{d}}

\newcommand{\bes}{\begin{equation*}}
\newcommand{\ees}{\end{equation*}}

\def\Ld{\mathcal{L}}
\def\R{\mathbf{R}}

\def\Qs {{\mathcal Q}}

\def\Ps {{\mathcal P}}

\title[Data-driven  non-Markovian closure models]{Data-driven  non-Markovian closure models}

\author[D. Kondrashov]{Dmitri Kondrashov}
\email{dkondras@atmos.ucla.edu}
\author[M.D. Chekroun]{Micka\"el D. Chekroun}
\author[M. Ghil]{ Michael Ghil}
\address[DK,MDC,MG]{Department of Atmospheric \& Oceanic Sciences and Institute 
of Geophysics \& Planetary Physics, 
405 Hilgard Ave,
Box 951565,
7127 Math Sciences Bldg.
University of California, 
Los Angeles, {CA 90095-1565,}  U.S.A.}

\address[MG]{Geosciences Department and Laboratoire 
de M\'{e}t\'{e}orologie Dynamique (CNRS and IPSL), \'Ecole Normale Sup\'{e}rieure, 
F-75231 Paris Cedex 05, FRANCE.}


\begin{document}
\maketitle

\begin{abstract}
This paper has two interrelated foci: (i) obtaining stable and efficient data-driven closure models by using a multivariate time series of partial observations from a large-dimensional system; and (ii) comparing these closure models with the optimal closures  predicted by the Mori-Zwanzig (MZ) formalism of statistical physics.
Multilayer stochastic models (MSMs) are introduced as both a generalization and a time-continuous limit of existing multilevel,  regression-based approaches to 
closure in a data-driven setting; these approaches include empirical model reduction (EMR), as well as more recent multi-layer modeling. It is shown that the multilayer structure of MSMs can provide a natural Markov approximation to the generalized Langevin equation (GLE) of the MZ formalism. 

A simple correlation-based stopping criterion for 
an EMR-MSM model is derived to assess how well it approximates the GLE solution. Sufficient 
conditions are derived on the structure of the nonlinear cross-interactions between the constitutive layers of a given MSM to guarantee the existence of a global random attractor. This existence ensures that no blow-up can occur for a broad class of MSM applications, a class that includes non-polynomial predictors and nonlinearities that do not necessarily preserve quadratic energy invariants. 

The EMR-MSM methodology is applied to a conceptual, nonlinear, stochastic climate model of coupled slow and fast variables, in which only slow variables are observed. It is shown that the resulting closure model with energy-conserving nonlinearities efficiently captures the main statistical features of the slow variables, even when there is no formal scale separation and the fast variables are quite energetic.

Second, an MSM is shown to successfully   reproduce the statistics of a partially observed,  Lokta-Volterra  model of population dynamics  in its chaotic regime. The challenges here include the rarity of strange attractors in the model's parameter space and the existence of multiple attractor basins with fractal boundaries. The positivity constraint on  the solutions' components replaces here the quadratic-energy--preserving constraint of fluid-flow problems  and it successfully prevents blow-up. 
\end{abstract}

\section{ Introduction and Motivation} 
\label{sec:intro}

\subsection{Background}
\label{ssec:background}

Comprehensive dynamical climate models aim at simulating past, present
and future climate and, more recently, at predicting it.  These
models, commonly known as general circulation models or global climate
models (abbreviated as GCMs in either case) represent a broad range of
time and space scales and use a state vector that is typically constituted  by several millions of
variables.

While detailed weather prediction out to a few days does require such
high numerical resolution, 
climate variability on longer time scales is dominated by large-scale patterns, which may require only a few appropriately selected modes  for their simulation and prediction \cite{GhilRob2000}. For a specific range of frequencies and targeted variables,  one may try to formulate low-order models (LOMs) for these purposes. Such models 
must account in an accurate way (i) for linear and nonlinear self-interactions between 
{a judiciously selected set of} resolved high-variance climate components; and (ii) for the cross-interactions between
the resolved components and the large number of unresolved ones. Although the present article is motivated  primarily
by the need for LOMs 
in climate modeling,  similar issues arise in many other areas of the physical and life sciences, and LOMs are becoming a key tool in various disciplines as diverse as astrophysics \cite{Stellar'05}, biological neuronal modeling  \cite{bachar2013stochastic}, molecular dynamics  \cite{noe2008transition,prinz2011markov} or pharmacokinetic and pharmacodynamic modeling \cite{donnet2013review}.

In climate dynamics, the prediction of the El Ni\~no-Southern Oscillation (ENSO) has attracted  increased attention during the past decades,  since ENSO constitutes the dominant mode of interannual climate variability, with major global impacts on temperature, precipitation, tropical cyclones, and human health \cite{Cam07P2, Kovats2003, Ropelewski87}; it has even been argued recently to have potential impacts on civil conflicts \cite{ENSOConflict2011}. 
The earliest successful predictions of ENSO were made using a dynamical
model governed by a set of coupled partial differential equations (PDEs)~\cite{CaneZebiak1986} that 
was  itself a highly reduced model by today's standards. Subsequently, stochastically driven linear LOMs \textemdash\, 
based either on observational data \cite{Penland_MWR89, PenlandSardeshmukh_JCL95,WinklerEtAl_JCL01} or 
on dynamical model simulations \citep{Xue94} \textemdash\, have been
used for ENSO predictability studies 
as well as for real-time forecasting \citep{PenlandMagorian1993}.  Nowadays, modeling 
ENSO by such LOMs can be considered as a significant success story, although predicting its extremes is still a challenge;  a recent survey 
on real-time prediction skill  of state-of-the art 
statistical ENSO models 
compared to comprehensive dynamical climate models  is given in \cite{iri12}.

Real-time ENSO predictions 
based on the {\it Empirical Model Reduction (EMR)} method 
introduced in \cite{kkg05,kkg05_enso} 
have proven to be highly competitive.\footnote{
Barnston and colleagues \cite{iri12} analyzed two dozen ENSO multi-model real-time predictions coordinated by Columbia University's International Research Institute for Climate and Society (IRI) over the 2002--2011 interval and concluded that the ``UCLA-TCD prediction (ensemble mean) 
has the highest seasonally combined correlation skill among the statistical models exceeded
by only a few dynamical models [...] as well as one of the smallest RMSE.'' Note that the former is based on models with a few tens of variables, while the latter have many millions of variables.}
 Within climate dynamics, the EMR  methodology has been successfully applied to the modeling of  many multivariate time series on  different time scales, 
{whether} arising in observed air--sea interactions in the Southern Ocean \citep{KravtsovEtAl_OS11}, in the identification and predictability analysis of nonlinear circulation regimes in atmospheric models of intermediate complexity \citep{kkg06,KondrashovKravtsovGhil_JAS10,Peters_etal12},  in the modeling of the Madden-Julian Oscillation (MJO) \cite{kcg_13MJO} or in the stochastic parameterization {of subgrid-scale mid-latitude processes \citep{kir}.}  

In these successfully solved {climate} problems, the key 
{ingredient} to modeling and predicting the dynamics of the macroscopic, observed variables from partial and incomplete observations of large systems is the appropriate use of some pre-specified 
self- and cross-interactions 
{between the} macroscopic variables, supplemented by some auxiliary, hidden variables.
In an EMR model, these interactions are typically chosen to be quadratic or linear and the associated hidden variables are arranged into a ``matrioshka'' of layers. Each supplementary layer in this ``matrioshka'' includes a hidden variable that is 
less auto-correlated than 
the one introduced in the previous layer, until some decorrelation criterion is reached; see  \ref{App1}. 
In practice, the unknown coefficients at the main level and the additional ones are learned by means of multilevel regression techniques; we refer to \citep{kkg09_rev}  for a review of the EMR methodology and a comparison with other model-reduction methodologies.

Quite recently, a couple of papers \cite{Majda_Yuan12,Majda_Harlim2012,MMH} 
pointed out the possibility of undesirable behavior  in EMR 
models, and proposed to add 
energy-preserving constraints 
on the quadratic terms in order to prevent  such behavior. r.  
The idea of adding constraints to prevent blow-up in an EMR formulation should clearly not be limited to the class of models that, in the absence of dissipation, possess quadratic energy invariants: 
many models in climate dynamics do not possess such invariants, e.g. models of the sea surface temperature (SST) field have energy terms that are linear in temperature, rather than quadratic. On the other hand, 
in certain situations that occur, for instance, in population dynamics \cite{smale1976differential, hirsch1982systems, hirsch1985systems,hirsch1988systems,hirsch1990systems, smith2008monotone} or chemical kinetics \cite{rossler1976chaotic,rossler1977chaos,segel1989quasi,maas1992simplifying,kalachev2001reduction}, while the nonlinearities might still 
be quadratic, introducing such constraints might actually be counter-productive; see 
Section \ref{sec_pop_dyn}.  The work of \cite{Majda_Harlim2012,MMH} 
proposed multi-level regression (MLR) models that did allow for quadratic interactions between the observed and some of the hidden variables and showed that these interactions --- given quadratic energy invariants in the flow models to which they were applied --- result in a stable, well-behaved reduction of the full flow models. As their very name indicates, though, the fundamental feature of the MLR models is still the multilevel structure proposed a decade ago
in the original EMR formulation.

The associated hidden variables in the original EMR formulation depend, due to this 
multilevel structure, on the past of the observed variables and bring therefore memory effects into the resulting low-dimensional stochastic models, in a fashion that is reminiscent of the closure models 
in the Mori-Zwanzig (MZ) formalism of statistical mechanics \cite{zwanzig2001nonequilibrium} or  
of related optimal prediction methods \cite{Chorin_MZ,Chorin_Hald-book}. 
The latter methods also deal with the problem of predicting just a few relevant variables but from a different perspective  than the EMR one, i.e. when the equations of the original, full system are available. The connection between the EMR formulation and MZ-type formalisms was first pointed out and illustrated by a simple example in the supporting information of \cite{CKG11}.


\subsection{{ Outline}}
\label{ssec:outline}

The background of this paper is thus provided by (a) the success of the EMR methodology in the modeling  and real-time prediction of spatio-temporal climate fields; (b) the recent criticisms in \cite{Majda_Yuan12,Majda_Harlim2012,MMH} of potential vulnerabilities in the original version of this methodology; and 
(c) its relationships with the closure methods  suggested by the MZ formalism. {The purpose of the paper is, therefore, (i) to generalize further both the original EMRs and the MLR models in \cite{Majda_Yuan12,Majda_Harlim2012,MMH}; (ii) to provide}
a mathematical analysis of data-based EMR models in their continuous-time limit and of the generalizations thereof; and (iii) to illustrate the insights and additional tools thus obtained by two simple applications. 

We call the generalized and rigorously studied continuous-time limit of EMRs
multilayer stochastic models (MSMs). In Sec.~\ref{sec:emr_form}, we formulate the closure problem in the presence of partial observations and consider EMR models as a candidate solution to this problem. In Sec.~\ref{Sec_MSM}, we introduce MSMs and show that an MSM can be written as a system of stochastic integro-differential equations (see Proposition \ref{Main_Prop}).  This system can lead in practice to a good approximation of the generalized Langevin equation (GLE) of the MZ formalism,  denoted here by \eqref{Eq_MZ_GLE} and studied in Sec.~\ref{MSM-MZ}. In this section, it is shown that the closure obtained by the GLE is theoretically optimal, given a time series of data, rather than a known master equation. Lemma  \ref{Main_lemma} supports this statement, subject to the appropriate ergodic assumptions and assuming an infinitely long multivariate time series of partial observations. 
The difference 
between the standard way of building the GLE and an approach based on averaging along trajectories, { as in} Lemma ~\ref{Main_lemma}, is similar to the Eulerian versus the Lagrangian viewpoint in fluid mechanics. 

 Sections~\ref{Sec_MSM} and \ref{MSM-MZ} are fairly theoretical and the hasty reader who might be more motivated by the applications can skip these two sections at a first reading and return to them later, after seeing the usefulness of their theoretical results in Secs.~\ref{Practical_MZ}--\ref{sec_pop_dyn}.   In Sec.~\ref{Practical_MZ}, practical issues of applying the results of Secs.~\ref{sec:emr_form}--\ref{MSM-MZ} to deriving accurate and stable EMRs are considered.
In particular, we will see that, under certain circumstances, an MSM can be understood as a Lagrangian approximation of the GLE; see { also} Proposition~\ref{Main_Prop}. 
 Numerical results for a conceptual stochastic climate model proposed in \cite{Majda_etal2005} are presented in Sec.~\ref{sec:Climate_ex}. These highly satisfactory results demonstrate, among other things, that the $\eta$-test formulated in Sec.~\ref{Practical_MZ} for the last-level residue of an MSM does provide quite an efficient  
criterion for the degree of approximation of the GLE solution by the appropriate MSM.

In Sec.~\ref{Sec_MSM} we also derive conditions on the cross-interactions between the constitutive layers of a given MSM  that guarantee the existence of a global random attractor. {This existence ensures  
that no blow-up  can occur for a broad class of} MSMs that generalize the class of 
EMR-like models used so far, including but not restricted to the MLR models of \cite{Majda_Yuan12,Majda_Harlim2012,MMH}; see Theorem \ref{Main_thm}. This class 
includes non-polynomial predictors and {nonlinearities that do not necessarily preserve quadratic} energy
{invariants, such as assumed in \cite{Majda_Harlim2012,MMH}; see  Corollary \ref{Main_corr}. 
The latter results are illustrated in Sec.~\ref{sec_pop_dyn} 
by solving a closure problem} arising in population dynamics that possesses merely linear and quadratic terms, but requires  a very different set of constraints  to prevent blow-up {of the reduced model.}

{Finally, four appendices provide further details on EMR stopping criteria, on real-time prediction using an MSM, on practical aspects of energy conservation, and on the interpretation of MSM coefficients.}


\subsection{Multilayer stochastic models (MSMs) and integro-differential equations}
\label{ssec:IDEs}

{In this subsection, we take a detour into the deterministic literature of integro-differential equations that will shed some further light on} the ability of an MSM to 
 provide an efficient closure model based on partial information on the full model, as derived from a time series.  
The parallels drawn herein between the two situations  yield a broader perspective on the {role of an MSM's multilayer structure} 
with respect to its representation as a system of stochastic integro-differential equations, 
cf. Proposition \ref{Main_Prop} below. 

The present remarks demonstrate the underlying relationships between multilayer systems of ordinary differential equations (ODEs) and systems of integro-differential equations, and 
help  
one understand why the multilayer structure of an MSM is essential 
{in constructing a class of stochastic differential equations (SDEs)} susceptible to approximate a GLE.
{These observations} are  actually rooted in older {mathematical ideas 
from the 
study of models that involve} distributed delays {\cite{Fargue,worz1978global}; such 
models arise} in theoretical population dynamics 
{and in the modeling of materials with memory \cite{dafermos1970asymptotic, Chek_al11_memo, chekroun_glatt-holtz}, as well as in climate dynamics \cite{Bhatt82, Roques2014ebmm}.}

Motivated by these remarks, we consider
now the following system of integro-differential equations
\be\label{Eq_integlogistic}
\frac{\d x_i}{\d t}=x_i\Big(b_i+\sum_{j=1}^n a_{ij}x_j +\sum_{j=1}^n\gamma_{ij} \int_{-\infty}^t g_{ij}(t-s)x_j(s) \d s\Big), \; 
 i = 1, \ldots, n;
\ee
this system models  the population dynamics of a community of $n$ interacting species, where $x_i$ denotes  the population density of the $i$-th species, $b=(b_1,b_2, ... ,b_n)$ the vector of intrinsic population growth
rates, $A=(a_{ij})$ and $\Gamma=(\gamma_{ij})$ denote the  interaction matrices, and $g_{ij}$ the memory kernels  that describe  the present response of  the {\it per capita} growth rate of a species $i$ to historical population densities  $x_j$.  
Volterra  proposed such a system of integro-differential
equations to describe an ecological system of interacting species and investigated 
 it for $n=2$ \cite[Chap.~IV]{volterra1931leccons}.
 
We wish to show how  system~\eqref{Eq_integlogistic} can be recast into a system of ODEs,  and assume for simplicity at first that \eqref{Eq_integlogistic} takes the 
form,
\bea\label{Eq_integlog2}
\frac{\d x_i}{\d t}&=x_i\Big(b_i+\sum_{j=1}^n a_{ij}x_j\Big), \; 
1 \le i \neq p \le n; \;  \text{and} \\
\frac{\d x_p}{\d t}&=x_p\Big(b_p+\sum_{j=1}^n a_{pj}x_j+\gamma_{pm} \int_{-\infty}^t g_{pm}(t-s)x_m(s) \d s\Big),
\eea
for some $p$ and $m$ in $\{1,\cdots,n\},$ 
i.e., that only a single equation exhibits memory effects. Furthermore, 
the memory kernel $g_{pm}$ is assumed to be given by the
Gamma distribution
 \be
 g_{pm}(t)=F_k(t)=\frac{{\alpha}^{k}}{(k-1)!}t^{k-1}e^{-\alpha t},
 \ee
 for some $\alpha>0$ and some positive integer $k\geq 1$. 
 
The key step is to note the recursion relation 
 \be
 \frac{\d F_k}{\d t}=\alpha F_{k-1}- \alpha F_k
 \ee 
 and to introduce the additional $k$ new variables $r_j$, with
 \be
 r_j(t) = \int_{-\infty}^t x_m(s)F_j (t-s) \d s, \; 
 j = 1, \ldots, k.
 \ee
By differentiation we obtain that these auxiliary variables 
obey the following system of ODEs, 
 \bea
 \frac{\d r_1}{\d t}&=\alpha (x_m -r_1), \\
 \frac{\d r_j}{\d t}&=\alpha (r_{j-1}-r_j), \;  j = 2, \ldots, k.
 \eea
This system is driven by $x_m$. More precisely, the dynamics of the auxiliary variable $r_1$  is directly slaved 
to that  of $x_m$, while the other $r_j$-variables are indirectly slaved to $x_m$, since each variable 
$r_j$ in a given layer $2\leq j \leq k$ interacts with $r_{j-1}$ in the previous layer, thus sharing a multilayer structure reminiscent of the one  in the original  EMR formulation \cite{kkg05,kkg05_enso}.   
 
These remarks allow us to recast the system  of integro-differential equations  \eqref{Eq_integlog2} 
as the following system of ODEs:
\bea\label{ODE_Lotka}
\frac{\d x_i}{\d t}&=x_i\Big(b_i+\sum_{j=1}^n a_{ij}x_j\Big), \;  1 \le i \neq p \le n, \\
\frac{\d x_p}{\d t}&=x_p\Big(b_p+\sum_{j=1}^n a_{pj}x_j+\gamma_{pm}r_k\Big),\\
 \frac{\d r_j}{\d t}&=\alpha (r_{j-1}-r_j), \; { j = 2, \ldots, k},\\
 \frac{\d r_1}{\d t}&=\alpha (x_m -r_1).\\
\eea 
The expansion procedure 
outlined above for the single memory effect in the simplified system of Eq.~\eqref{Eq_integlog2} can obviously be carried out for any number of memory terms in any number of equations, subject to the addition of a suitable
number of linear equations;  hence it can also be applied to the general  
system of integro-differential equations in Eq.~\eqref{Eq_integlogistic}.

One concludes that, for such a system integro-differential equations --- if the kernels are weighted sums of Gamma distributions, or more generally, if these kernels are solutions to a linear system of ODEs with constant coefficients --- then the original system can be transformed into  a system of ODEs. This transformation is 
known as the ``linear-chain trick'' \cite{Fargue, macdonald1978time, smith2011}. Of course, it is important to be able to go in the other direction as well. If one
finds an interesting solution of the ODE system \eqref{ODE_Lotka}, 
e.g. a periodic solution,
then one wants to know if  it does solve \eqref{Eq_integlog2} as well. In fact, it is not difficult to prove that any solution of \eqref{ODE_Lotka} that is bounded on the entire real line
is also a solution of the integro-differential equation \eqref{Eq_integlog2}; see 
~\cite[Prop.~7.3]{smith2011}. Since the resulting system of ODEs~\eqref{ODE_Lotka} does not involve the knowledge of the past of the $x_i$-variables,  one can say that  
a ``Markovianization'' of the original system has been performed by 
suitably augmenting 
the number of variables; 
 this augmentation procedure 
is actually well-known in the rigorous study of  systems with distributed delays, such as 
those that arise in the modeling of materials with memory, for instance, 
cf.~\cite{dafermos1970asymptotic, Chek_al11_memo, chekroun_glatt-holtz} and references therein.

This detour via a class of systems of integro-differential equations provides 
some general guidance on how to
``Markovianize'' a broad class of GLEs 
of the type predicted by  MZ closure procedures. Such GLEs 
take necessarily the form of systems of stochastic integro-differential equations.  It is natural, then, to seek
approximations to such MZ closures in the form of an augmented system of SDEs whose main, observed variables are supplemented by  appropriate auxiliary hidden variables, and one expects the latter variables to interact with the main ones and among themselves in a fashion  suggested by the ODE system~\eqref{ODE_Lotka}.

Of course, the corresponding interactions 
 have to take a specific form, depending on the applications. 
{For} dissipative systems, it is shown below that --- given a natural energy that has to {be dissipated ---}
 simple estimates allow 
{one to identify permissible} interactions that ensure the existence of dissipative MSMs; see Theorem \ref{Main_thm} and Corollary \ref{Main_corr}. Within this class of interactions, Proposition~\ref{Main_Prop} 
{ensures that} a simple correlation-based criterion formulated in Section~\ref{Practical_MZ} {does} address the problem of {approximating the} GLE by such MSMs.

The approach proposed in this article 
 complements, therewith, more traditional techniques for the Markovian approximation of the GLE. 
Typically, the  latter approaches rely on a continued-fraction expansion of the Laplace transform 
of the autocorrelation functions of the noise in the GLE, as introduced by Mori \cite{mori1965continued}, or on related approximations by rational functions of  linear GLE memory kernels \cite{kupferman2004fractional,kupferman2004fitting}.
Still, the applicability of these Markov approximation techniques 
is limited by relying on rather restrictive assumptions, 
such as systems with a separable,
quadratic Kac-Zwanzig Hamiltonian 
\cite{kupferman2004fractional} 
 or linear kernels, although non-Gaussian noise in the GLE 
is allowed \cite{kupferman2004fractional}. These restrictions led some authors to conclude that 
MZ models with 
{linear, finite-length} kernels form a subclass of 
autoregressive moving average (ARMA) models \cite{horenko2007data,niemann2008usage}. 

As shown in the body of this paper, our approach 
to the derivation of Markov approximations to the GLE goes beyond these limitations and allows for 
nonlinear kernels as well as for non-Gaussian noise;  it applies, {furthermore,} to a broad class of dissipative, rather than Hamiltonian systems. Nevertheless, the considerations in this subsection demonstrate 
the intuitive relevance of the multilevel structure inherent 
in the EMR methodology --- albeit initially designed from a different perspective \cite{kkg05,kkg09_rev} --- for the derivation of closure models from a multivariate time series of partial observations. 
This article  
shows, furthermore, that --- when suitably generalized --- the 
multilayer EMR methodology  provides an efficient means of deriving   
such closure models,   as well as facilitating their mathematical analysis.


\section{The closure problem from partial observations and its EMR 
solution}\label{sec:emr_form}


As discussed in Secs.~\ref{ssec:background} and \ref{ssec:outline} above, we are
motivated by the modeling of geophysical fluid flows --- as well as of more complex climate problems and of large-dimensional problems from other fields of science ---  based on a series of partial observations. We formulate below  the corresponding observation-based closure problem ($\mathfrak{P}$) and recall its EMR candidate solution, such as initially proposed in  \cite{kkg05,kkg05_enso}. Generalizations  of such an EMR solution are discussed in Section \ref{Sec_MSM} below.

One can consider the approach presented in this article as complementary to the derivation of { deterministic} nonlinear dynamics from 
{observations,} in the spirit  of Ma\~ne-Takens \cite{Mane1981, Takens1981}  
attractor reconstruction {by phase-space embedding of a time series:} instead of just trying to reconstruct the attractor from a single time series or from a multivariate one \cite{Broom_King1986a, Broom_King1986b, Ghil2002}, we attempt to actually write down equations that will produce a good approximation of the attractor, 
{including both} its geometry and invariant measure. 

Several approaches can be used for {this purpose.} Among them 
{are} the recent time-lagged polynomial sparse modeling technique 
{for the} nonuniform embedding of time series \cite{nichkawde2014sparse}, or the more traditional approaches based on ARMA models \cite{BoxJenkins70}. The 
nonlinear version { of the latter \cite{billings1989identification,lu2001new} is} somewhat closer to the EMR methodology, due to the combined  presence of noise and memory effects\footnote{Note 
that, in some sense, the EMR methodology can also be viewed as an extension of 
hidden Markov models (HMMs) \cite{Willsky2002Markov, Ihler2007graphical} or of artificial neural networks (ANNs) 
\cite{Hornik1989ann, Mukhin2014ann1, Mukhin2014ann2},
since the latter are generally nonlinear but do not involve the memory effects inherent in the EMR methodology; see \cite{kcg_13MJO, CKG11}.}.  However, the EMR models differ in their parameter estimation by the  top-to-bottom multilevel procedure  recalled below and subject to the stopping criterion described in~\ref{App1}.
As we will see from the theoretical results of Section \ref{Sec_MSM}, the multilevel structure intrinsic to EMR allows for a great flexibility in specifying various linear and nonlinear interactions between the main-level and hidden-level variables, in order to design MSM generalizations of EMR models; see Theorem \ref{Main_thm} and Corollary \ref{Main_corr} below.  The same multilevel structure allows us furthermore to relate MSMs to MZ models, as explained in Sections~\ref{MSM-MZ} and \ref{Practical_MZ}.  

As mentioned earlier, MZ models provide optimal solutions to closure problems such as the observation-based  problem ($\mathfrak{P}$) formulated below. As we will see, EMR or their MSM generalizations of Section \ref{Sec_MSM} provide approximate solutions to such optimal solutions; see Section~ \ref{sec:Climate_ex} and \ref{sec_pop_dyn} for applications.

\begin{itemize}
\item[($\mathfrak{P}$)]  Let $\{u(t_{k},\xi_j)\}$ be a set of discrete observations of 
{a given,} spatio-temporal, scalar field, where $\{t_k=k\delta t\}$\footnote{Here $\delta t$ denotes the sampling 
{interval} of the 
data. For the sake of simplicity, we consider in this article data that are 
{uniformly sampled} with a constant $\delta t$.}, and the grid points $\{\xi_j\}$ live in  
a spatial domain $\mathcal{D}$, which can be 
two- or three-dimensional.  The amount of data is  always assumed to be finite, {\it i.e.} $k\in\{1,...,N\}$ and $j\in \mathcal{J}$, where $\mathcal{J}$ is a finite set of  multi-indices; often, the data set is rather limited, and thus $N$ is not as large as we would like. 
The goal is to find a model that 
not only 
{describes} the evolution of the  
$\{u(t_{k},\xi_j)\}$ 
observed so far, but also 
{possesses} good prediction skills at future times, $k>N$, and at the same sites $\{\xi_j\}$.

\end{itemize}

When equations are available to describe the evolution of $\{u(t_k, \xi_j)\}_{j\in \mathcal{I}}$ in the domain $\mathcal{D}$, with card$(\mathcal{I})>>$ card$(\mathcal{J})$, this problem is related to closure problems of a type 
that has been widely studied in continuum physics, and a variety of approaches is available for dealing with such  problems \cite{Majda_turb}. In the general case,  {though,} no equations are available for 
the evolution of the full field $\mathcal{U}_{\mathcal{I}} := \{u(t_k, \xi_j)\}_{j\in \mathcal{I}}$, while the subfield $\mathcal{U}_{\mathcal{J}}:=\{u(t_k, \xi_j)\}_{j\in \mathcal{J}}$ is assumed to consist of {\it partial observations} from the field $\mathcal{U}_{\mathcal{I}}$ that 
contains subgrid information not resolved by the field $\mathcal{U}_{\mathcal{J}}$. 

The dynamics of the unobserved variables $\{u(t_k, \xi_j)\}_{j\in \mathcal{I}\backslash \mathcal{J}}$ is thus lacking and the main issue is to derive an efficient  parameterization of the interactions between the resolved and unresolved --- i.e.,  roughly speaking, the observed and unobserved --- variables in order to derive equations that model the evolution of the field $\mathcal{U}_{\mathcal{J}}$ with reasonable accuracy.  Furthermore, the scalar field $u$ may interact with other fields  that are not taken into account or not observed, which makes 
an accurate modeling of the field $\mathcal{U}_{\mathcal{J}}$ even more difficult. Typically, two-point statistics --- such as correlations or the joint probability density function (PDF) of the observed variables --- are 
used to assess the effectiveness and accuracy of the proposed closure model.

Besides these modeling aspects, the prediction requirement in the problem ($\mathfrak{P}$) above is clearly challenging, in theory as well as in practice. It is 
well known that an inverse model may approximate the data 
accurately over the training interval, while exhibiting poor prediction skill on the validation interval during hindcast experiments, and performing even worse in real-time prediction.
Nonetheless and as mentioned in Sec.~\ref{ssec:background}, the data-driven EMR procedure --- described in Steps 1--3 below--- has
proven to be 
{quite skillful in} real-time ENSO prediction \cite{iri12}.  For example,  although  the EMR-ENSO model of \cite{kkg05_enso} is based on a leading subset of principal components of the  
SST field that 
constitute only a small fraction of the total set of degrees of freedom of the ENSO phenomenon; still, this EMR-ENSO model ``has the highest seasonally combined correlation skill among the statistical models, exceeded by only a few dynamical models''; see footnote$^1$.

As initially formulated in \cite{kkg05}, after appropriate compression
of the available data, cf.~Step 1 below, the 
{key} idea consists 
in seeking 
forced-dissipative models of the form
\begin{equation}
\label{equ:EMR_simple}
\frac{\delta x_k}{\delta t} = -Ax_k+B(x_k,x_k)+ F+\mathrm{``noise."}
\end{equation}
Here $x_k = x(t_k)$ is the $d$-dimensional column vector of the relevant variables to be modeled, ${\delta x_k}/{\delta t}$ represents its rate of change   in time, $F$ is a column vector, $A$ is a square $d\times d$ matrix, and $B(x_k,x_k)$ accounts for the bilinear interactions that influence the evolution of ${\delta x_k}/{\delta t}$. The noise term accounts typically 
for the effects of the unresolved variables  --- i.e., of subgrid-scale processes or  of other unobserved variables  --- on the dynamics of the observed variables, as modeled by Eq.~\eqref{equ:EMR_simple}. 

Often, this noise is simply assumed to be Gaussian and white in time but, in practice, its amplitude may depend on the
fluctuating variable $x$ itself. It is, therefore reasonable to model such a noise as {\it state-dependent},
and allow for it to possess non-vanishing correlations in time as well as in space.
Actually, Markov representation theorems 
ensure the existence of such a state-dependent noise, as long as the original, possibly large-dimensional, deterministic, time-discrete, observed dynamical system,  
possesses a relevant invariant measure that is {\it physical} in the sense recalled in Section~\ref{MSM-MZ} below, cf. {Eq.~\eqref{Eq_phys_rev};} see also Corollary B in the Supporting Information of \cite{Chek_al13_RP}. In the EMR methodology, this state-dependent noise is modeled by means of successive regressions against the coarse observed variable $x$, until the time correlations of the residual noise 
 satisfy suitable decorrelation criteria.\footnote{For the sake of clarity, { the decorrelation criterion is} recalled in \ref{App1}. 
Roughly speaking, the residual noise { has to have} vanishing correlations at lag $\delta t$, i.e., at the sampling  
{interval} of the available data.  The successive minimizations, 
described in Steps 1--3 below, have been shown, in practice,   to reach such a limit, within a reasonable error, after solving only finitely many minimization problems; see \cite{kkg05, kkg05_enso, kkg06, kkg09_rev} and 
Remark \ref{Rem:scaling} below.}

The main 
steps of the EMR { procedure} can  be summarize as follows:
\begin{itemize}
\item[(i)] {\bf Step 1: Data compression.} Let $\mathcal{B}=\{E_1,...,E_d\}$ be an orthogonal basis\footnote{This basis is determined from the data; empirical orthogonal function (EOFs) are often used in EMR methodology, but other possibilities do exist \cite{kir, kkg09_rev}; see also \cite{Crommelin_Majda} for 
{a comparison} of different bases { used} in model reduction.} 
that spans a finite-dimensional Euclidean space of dimension $d$ with associated norm $\|\cdot\|$, so that the following orthogonal decomposition of $u(t_{k},\xi_j)$ onto $\mathcal{B}$ can be easily determined:
\begin{equation}
\label{equ:expand}
u(t_{k},\xi_j)=\sum_{i=1}^{d} x_i(t_k) E_i(\xi_j).
\end{equation}
In what follows, 
$x(t_k)$ is the $d$-dimensional column vector\footnote{If $\mathcal{B}$ denotes the set of EOFs, then the $x_i$'s correspond to the principal components.} $x(t_k):=(x^1(t_k),...,x^d(t_k))^{{\mathrm T}}.$   Given this decomposition, the problem ($\mathfrak{P}$) is now restricted to the modeling and eventual prediction of the time evolution of $x(t)$ or even of just a few of its components. A data-driven approach to this problem proceeds in the two following steps. 

\item[(ii)]  {\bf Step 2: Main-level regression.} 
Given the multivariate time series 
{$\{x(t_k): k = 1,\ldots,N\}$, one seeks} a deterministic column vector $F_N$, a square matrix $A_N$, and a quadratic form $B_N(x,x)$ 
to achieve the following  $l_2$ minimization:
\begin{equation}
\label{equ:argmin}
{[F_N,A_N,B_N]}=\underset{F,A,B}{\textrm{argmin}}  \Big( \sum_{k=1}^{N-1}\Big\| \frac{\delta x_k}{\delta t}-F+Ax_k-B(x_k,x_k)\Big\|^2\Big);
\end{equation}
{here} $\delta x_k:= x_{k+1}-x_k$, and  
{$x_k = x(t_k)$ represents the 
evolution in time} of $u(t_{k},\cdot)$ { as decomposed onto the basis $\mathcal{B}$}. 
Clearly,  $F_N$, $A_N$ and $B_N$ in Eq.~\eqref{equ:argmin} 
are meant to approximate $F$, $A$ and $B$ in Eq.~\eqref{equ:EMR_simple} so as to optimize in an $l_2$-sense the evolution of $x_k$ with respect to the data $u(t_{k},\xi_j)$. The subscript $N$ 
emphasizes the fact that  $F_N$, $A_N$ and $B_N$ are estimated from the finite-length, multivariate time series $\{x_k\}_{k\in \{1,...,N\}}$. 

\item [(iii)] {\bf Step 3: Multilevel regression.}  Let $ r_k^{(0)}:= r^{(0)}(t_k)\in \mathbb{R}^d$ be the $d$-dimensional regression residual associated with the $l_2$-minimization problem of { Step 2.}  
We seek now a $d\times 2d$ 
rectangular matrix $L^{(1)}_N$   cross-interactions, associated with the main-level residual $ r_k^{(0)}$ and { the main-level variable $x$},  so that:

\begin{equation}
\label{emu:level_1}
L_N^{(1)}=\underset{L\in \mathbb{R}^{d\times 2d}}{\textrm{argmin}} \Big( \sum_{k=1}^{N-1}\Big\| \frac{\delta  r_k^{(0)}}{\delta t}-L\big[(x_k)^T,( r_k^{(0)})^T\big]^T\Big\|^2\Big).
\end{equation}
 
This main-level step is followed by solving a sequence of $l_2$-minimization problems: for each $m$ of interest,
we seek a $d\times (m+1)d$ 
rectangular matrix $L^{(m)}$ of cross-interactions, which is given by:

\begin{equation}
\label{equ:higher_level}
L_N^{(m)} = \underset{L\in \mathbb{R}^{d\times (m+1)d}}{\textrm{argmin}} \Big( \sum_{k=1}^{N-1}\Big\| \frac{\delta  r_k^{(m-1)}}{\delta t}-L\big[(x_k)^T,( r_k^{(0)})^T,...,( r_k^{(m-1)})^T\big]^T\Big\|^2\Big).
\end{equation}
This recursive sequence of minimizations is stopped when, for some $m = p$, the residual  $r_k^{(p)} \delta t$ has vanishing autocorrelation at lag $\delta t$, according to  the stopping criterion described in \ref{App1}.

\end{itemize}

From a practical  point of view, 
the solution of the successive minimization problems described above is subject to the general problem of {\it multicolinearity}, which arises from finite sample size of the { data set \cite{kkg05}.} Multicolinearity can lead to {statistically} unstable estimates, in the sense that the latter can be very sensitive to small changes in the data set. Various regularization techniques 
exist to deal with this problem \cite{hastie2009linear};  they rely mainly on  penalizing the $l_2$-functionals in { Steps 1--3 above}, 
and they have been shown to be effective in 
yielding stable estimates for the EMR models of various climate 
fields \cite{kkg05, kcg_13MJO, kkg09_rev}.

Once the time-independent forcing term $F_N$, the square matrix $A_N$, the bilinear term $B_N$, and the rectangular matrices 
$\{L^{(m)}_N: m = 1, \ldots ,p\}$ have been determined, we are left 
with an EMR model of the form:
\be\label{Eq_EMR} \tag{{\it EMR}}
\boxed{
\begin{aligned} 
& \hspace*{1.4em}x_{k+1}-x_k=\big[-A_Nx_k +B_N(x_k,x_k)+F_N\big]\delta t +  r_k^{(0)}\delta t, \quad k\in\{1,...,N\},\\
& \hspace*{1.4em}  r_{k+1}^{(m-1)}- r_{k}^{(m-1)}= L^{(m)}_N\big[(x_k)^T,( r_k^{(0)})^T,...,( r_k^{(m-1)})^T\big]^T\delta  t+ r_k^{(m)}\delta t, \quad { 1 \le m \le p},
\end{aligned} 
}
\ee
{since, in general, $p \geq 1$. 
This EMR models} the dynamics of the multivariate time series { $\{x_k: k = 1, \ldots,N\}$} of length $N$ { and should be able to predict it for $k > N.$} 

Note that, for prediction purposes, the model defined by Eq.~\eqref{Eq_EMR} has to be initialized appropriately in the past, along with estimating the  {\it hidden} $ r^{(m)}$-variables from the {\it observed} $x$, as explained in  \ref{App3}.  For the moment we turn to {a natural continuous-time formulation of Eq.~\eqref{Eq_EMR} that will help  set up a framework in which the EMR methodology will 
become theoretically more transparent by benefitting from a natural MZ interpretation; see Secs.~\ref{Sec_MSM}--
\ref{Practical_MZ}.

{Assuming} that $F_N$, $A_N$,  $B_N$ and 
{$\{L^{(m)}_N: m = 1, \ldots ,p\}$} have converged once the amount of data gets large enough\footnote{in some appropriate weak sense 
\cite{billingsley2009convergence}.}, we are formally left with an MSM that becomes, in the  
limit of $\delta t\rightarrow 0$ and $N\rightarrow \infty$:
\be\label{Eq_MSM} 
\begin{aligned} 
& \hspace*{1.4em}\d x=\big[-Ax +B(x,x)+ F\big]\d t +  r^{(0)}\d t, \quad 
{ 0 < t \le T^*,} \\
& \hspace*{1.4em} \d r^{(m-1)}= L^{(m)}\big[(x)^T,( r^{(0)})^T,...,( r^{(m-1)})^T\big]^T\d t+ r^{(m)}\d t, \quad { 1 \le m \le p.}
\end{aligned} 
\ee
{Here $[0,T^*]$} corresponds to the time interval over which the data are available --- also known as the {\it training interval ---}  during which the MSM parameters of \eqref{Eq_MSM} have been estimated.  

Note that in the theoretical limit above, the minimization functionals appearing in Steps 1--3 
have to be replaced by their continuous analogues; this just means replacing the $l_2$-functionals by their $L^2(0,T^*)$ version. {The rigorous justification of such a limit is of course a nontrivial task 
for general data. It implies in particular going beyond traditional finite-time convergence theory for numerical methods applied to SDEs \cite{KP92}.  SDEs such as those considered in \eqref{Eq_MSM} involve
only locally Lipschitz coefficients, 
{because of} the bilinear term $B$,  
{and thus} require more sophisticated techniques to prove strong convergence results of the solutions of  \eqref{Eq_EMR}\footnote{ in which $F_N$, $A_N$,  $B_N$ and $L^{(m)}_N$ would be replaced by $F$, $A$,  $B$ and  $L^{(m)}$, respectively.} to  
the solutions of \eqref{Eq_MSM}, when the former { system} is interpreted as an Euler-Maruyama discretization of the latter. 
{Such a convergence can still} be guaranteed, { provided the 
$n^{\rm{th}}$ moment} of the exact and the numerical solutions are bounded for some { $n \ge 3$;} see \cite{higham2002strong,mao2006stochastic}.  
{Leaving such} considerations aside in the present article, 
{we focus here} on  theoretical { MSM} aspects that will turn out to be useful from a modeling point of view,  
{as shown in Secs.~\ref{Sec_MSM}
--\ref{Practical_MZ}.}

\br\label{Rem:scaling}
Note that convergence results such as those mentioned above have been 
{ obtained} for SDEs driven by white noise. 
The system~\eqref{Eq_MSM} is a system of random differential equations (RDEs),  in which the last-level noise term $ r^{(p)}$ is not necessarily white according to the stopping criterion recalled in \ref{App1}. Equation~\eqref{Eq_MSM} can, however, turn out to be well approximated by a genuine system of SDEs driven by a Wiener process $W$. Let us assume that the residual noise $ r^{(p)}$, obtained from the last level of 
{ Step 3} above, obeys the following scaling law as $\delta t \rightarrow 0:$
\be\label{Eq_approx_white}
 r^{(p)} \sqrt{\delta t} \underset{\delta t \rightarrow 0}\sim \mathcal{N}(0,Q),
\ee 
where $\mathcal{N}(0,Q)$ denotes the normal distribution with  zero mean and covariance matrix  $Q$.
Then  $ r^{(p)}\delta t$  
{ is} a natural approximation of $Q\; \d W$ as $\delta t \rightarrow 0$, where $\d W$ denotes a $d$-dimensional white noise process. 

We emphasize,  though, that  the scaling law \eqref{Eq_approx_white} may be violated by certain data sets and 
that other scaling laws may apply to the stochastic process obtained in the limit of $\delta t \rightarrow 0$.
These laws could lead, for instance, to MSMs driven by L\'evy $\alpha$-stable processes, as 
argued for certain climate fields 
~\cite{penland_levy}. 
\er

When the scaling of \eqref{Eq_approx_white} is violated, one may wish to consider other types of interactions than quadratic or linear in the learning stage of \eqref{Eq_EMR}.  Doing so may better explain the state dependence of the main-level residual noise,  and lead naturally to the more general class of MSM considered hereafter.


\section{MSMs: general formulation and random attractors}
\label{Sec_MSM}

Motivated by the interest in more general problems than those posed by geophysical fluid flows, we identify here a general class of MSMs that possess a global random attractor, 
{whose existence prevents} the blow-up of model solutions. This class extends the one proposed recently in \cite{Majda_Harlim2012,MMH} and {does not require} the inclusion of energy-preserving nonlinearities to ensure the existence of such an attractor.

\subsection{MSMs and global random attractors}
\label{Sec_MSMs_att}
A natural extension of the  MSMs introduced { in Sec.~\ref{sec:emr_form}} 
can be written in compact form as follows:
\be\label{Eq_MSMa} 
\begin{aligned} 
& \hspace*{1.4em}
\d x=F (x) \d t + \Pi  r \;  \d t,\\
& \hspace*{1.4em} 
\d r = ( C x -D r ) \d t + \Sigma \,\d W_t.
\end{aligned} 
\ee
Here {$ r=( r_1,..., r_p)^{\rm T}$} is a $p\times q$-dimensional vector 
{with components that are $q$-dimensional each, such that $pq\geq d$, while} $C$ and $D$ are respectively $pq \times d$ and  $pq \times pq$ matrices, and $\Pi$ is the orthogonal  projection from $\mathbb{R}^{pq}$  onto $\mathbb{R}^{d}$. The matrix $\Sigma$ is a $(pq)\times q$-rectangular matrix whose last $q$ rows are equal to a positive-definite matrix $Q$ and zero elsewhere.

Another, even more general extension can be written in the following form:
\be\label{Eq_MSMgen} \tag{{\it MSM}}
\boxed{
\begin{aligned} 
& \hspace*{1.4em}
\d x=f_1(x) \d t  + g_1(u) \d t \;  \\
& \hspace*{1.4em}
\d y=f_2(y) \d t  +g_2 (u) \d t +  \Pi  r  \d t, \\
& \hspace*{1.4em} 
\d r = h(u, r)\d t + \Sigma \, \d W_t.
\end{aligned} 
}
\ee
{Here $u=(x,y)\in \mathbb{R}^d\times \mathbb{R}^{d'}$ and $\Pi$ 
is} the orthogonal projection from $\mathbb{R}^{pq}$ onto $\mathbb{R}^{d'}$; {for simplicity,}  
the functions $f_1,f_2,g_1,g_2$ and $h$ are assumed to be continuous and locally Lipschitz 
{in $\mathbb{R}^d\times \mathbb{R}^{d'}$.
The $d \times d'$-dimensional vector $y$ here} plays the role of a hidden variable, similar to {the $d$-dimensional vector} $r_0$ introduced in \eqref{Eq_MSM}, except that the  $y$-variable is 
allowed {now} to contribute nonlinearly to the dynamics of $x$ via the vector field $g_1$ and {that it} is not necessarily of the same dimension as 
$x$, i.e.~$d'$ may differ from $d$. 
{Likewise,} nonlinear effects are allowed this time for 
{both the} hidden variables $y$ and $r$. 
{The generalization of \eqref{Eq_MSM} via $d \neq d'$ and nonlinear effects in the second equation of \eqref{Eq_MSMgen} was} proposed recently in \cite{Majda_Harlim2012,MMH}, {but $h$ there} was assumed to be linear.  

Obviously, the cross-interactions between the observed and hidden variables, 
modeled { here} by the terms { $g_1, g_2$} 
and $h$, as well as the self-interactions  
contained in $h$ and 
{carried} by $f_1$ and $f_2$, have to be designed properly so that the system remains stable. {Ref. \cite{Majda_Harlim2012} has addressed this problem} for the class of {quadratic-energy--preserving} nonlinearities, {while keeping $h$ linear, by applying} 
geometric ergodic theory to the Markov semigroup associated with such systems.  
In this section, we adopt an alternative approach and address the problem of existence of random attractors for  the more general class of inverse models governed by \eqref{Eq_MSMgen}.

\br\label{Rem:Levy}
Note that random attractors may exist while the H\"ormander condition used in \cite{Majda_Harlim2012} is violated. {Such a violation} may arise, for instance, for stochastically perturbed system of ODEs  {that exhibit a strange attractor and support a Sinai}-Bowen-Ruelle measure $\mu$, 
when the noise acts transversally to the support of $\mu$, 
along the stable Oseledets subspaces \cite{eckmann_ruelle}. In other words, random attractors may still exist for SDEs driven by degenerate noise and { our present approach may  
be} viewed as complementary 
{to that of} \cite{Majda_Harlim2012}; see Theorem \ref{Main_thm} and Corollary \ref{Main_corr} {below.} Furthermore, 
the random-attractor arguments used here are not limited to the case of white noise and can be used to ensure the existence of non-degenerate solutions for MSMs driven by more general noises \cite{gess_al11}, as proposed in \cite{penland_levy}; see also Remark \ref{Rem:scaling}. 
\er

We proceed by transforming the system of SDEs in \eqref{Eq_MSMgen} into a system of RDEs that is more amenable to the application of energy estimate 
or Lyapunov function 
{methods, which} we use hereafter; see \cite{schenk1998random}. Mathematically, the benefits of dealing with the transformed system instead of the original one relies on the  H\"older  property of {Ornstein-Uhlenbeck processes,} which allows  one to rely on the  classical theory of non-autonomous evolution equations, parameterized by the realization $\omega$ of the driving noise; see 
\cite[Lemma 3.1  and Appendix A]{CLW14_vol1}.  To  perform this transformation,  we consider 
the following auxiliary $pq-$dimensional Ornstein-Uhlenbeck process, obtained  as the stationary solution of
\be\label{Eq_OU_gen}
\d z = -A_0 z\, \d t + \Sigma \, \d  W_t,  
\ee
where $A_0$  is a positive-definite $pq \times pq $ matrix, that will be 
determined later on. The change of variables $ r \leftarrow  r -z$ transforms the system \eqref{Eq_MSMgen} into the aforementioned system of RDEs:
\be
\begin{aligned} \label{genMSM_transf}
& \hspace*{1.4em}
\frac{\d x }{\d t}=f_1(x) + g_1(x,y) , \;  \\
& \hspace*{1.4em}
\frac{\d y}{\d t}=f_2(y)+g_2 (x, y) + \Pi ( r +z) ,\\
& \hspace*{1.4em} 
\frac{\d r}{\d t}= h(u,r)+ A_0 z ,
\end{aligned} 
\ee
and let us assume, 
{ for simplicity,} that  
\be\label{Eq_decomp_h}
h(u, r)=-D  r + \widetilde{h}(u, r),
\ee
with $D$  positive definite and $A_0 = D$. 

Following the Lyapunov function techniques 
in \cite{schenk1998random} for the existence of random attractors,  
we assume that there exists a map $V: \mathbb{R}^{d+d'} \rightarrow
\mathbb{R}^+$, 
which is continuously differentiable and has the property that pre-images of bounded sets are bounded, e.g.~$V$ is polynomial. Furthermore, we require that $V$ verify the following  conditions:
\be 
\label{cond1}\tag{H$_1$} 
\exists \; \alpha, \beta >0 \; : \;   \langle f_1(x), \nabla_{x} V (u)\rangle +\langle f_2(y), \nabla_{y} V (u)\rangle+ \alpha \cdot V(u) \leq \beta,
\; \forall \; u \in \mathbb{R}^{d+d'}, 
\ee

\be
\label{cond2}\tag{H$_2$} 
\exists \; a, b >0 \; : \;  \langle g_1(u), \nabla_{x}  V (u)\rangle + \langle g_2(u), \nabla_{y}  V (u)\rangle\leq   a   V(u)+b,
\; \forall \; u \in \mathbb{R}^{d+d'}, 
\ee

\be
\label{cond3}\tag{H$_3$}
\langle \Pi r, \nabla_{y} V(u)\rangle +\langle \tilde{h}(u,r), r\rangle \leq m_1 V(u) +m_2\|r\|^2,
\ee

\be
\label{cond4}\tag{H$_4$}
\|\nabla_{y} V(u)\|^2 \leq \gamma V(u) + C,
\ee
where $\langle f, g \rangle$ is the Euclidean inner product in $\mathbb{R}^{d+d'}$.

For a given realization $\omega$ of the noise, {we denote by $V (t,\omega)$ the 
values} of $V$ along a trajectory $u(t,\omega)=(x(t,\omega),y(t,\omega))$,
{and obtain, on the} one hand, from the $x$- and $y$-equations in \eqref{genMSM_transf}, 
\be\label{EE1}
\d V (t,\omega)= \Big(\sum_{i=1}^2  \langle
 f_i(x_i (t,\omega)), \nabla_{x_i}  V(t,\omega) \rangle  + \langle
g_i(x(t,\omega)), \nabla_{x_i}  V (t,\omega)\rangle  \Big) \d t,
\ee
{where we used the notations $x_1=x, \; x_2=y$;} 
on the other, after multiplication of the $ r$-equation by $ r$ in \eqref{genMSM_transf}, { we get}
\be\label{EE2}
\frac{1}{2}\d \|r\|^2 +\lambda \|r\|^2 \d t\leq  \Big( \langle \Pi  (r +z), \nabla_{y} V(t,\omega)\rangle +\langle \tilde{h}(u,r), r\rangle  \Big) \d t,
\ee
where $\lambda>0$ denotes the smallest eigenvalue of $D$.

By adding \eqref{EE1} and \eqref{EE2},  by using \eqref{cond1}--\eqref{cond3},  and by applying the $\epsilon$-Young inequality 
to the term $ \langle \Pi_2 z, \nabla_{y} V(t,\omega)\rangle$, we obtain} that, for any $\epsilon >0,$ there exists $C_{\epsilon}>0$ { such that}
  \bea
 \d V(t,\omega) + \frac{1}{2}\d \|r\|^2 + (\alpha V (t,\omega)+ \lambda \|r\|^2 )\d t & \leq \Big((a+m_1 +\epsilon \gamma) V (t,\omega) + \\
 & \beta +b + m_2 \|r\|^2 + \epsilon C + C_{\epsilon} \| \Pi_2 z(t,\omega)\|^2 \Big) \d t. 
 \eea

This inequality, in turn, leads to 
\be
 \d \widetilde{V}(t,\omega) +\kappa \widetilde{V}(t,\omega) \d t \leq (\widetilde{\beta} + \xi(t,\omega)) \d t,
\ee
with $\widetilde{V}(t,\omega) := V(t,\omega) + \frac{1}{2} \|r\|^2$, $\widetilde{\beta}=\beta +b+ \epsilon C,$ { and} 
\be\label{defkappa}
\kappa = \min(\alpha -a -m_1-\epsilon_1 \gamma ;2( \lambda -m_2)) 
\ee
for some $\epsilon_1>0$,
where 
\be 
 \xi(t,\omega)=C_{\epsilon_1} \| \Pi z(t,\omega)\|^2.
\ee
 
 Since the Ornstein-Uhlenbeck process $z$ is stationary with respect to the canonical shift $\theta_t$, so is the process $\xi(t,\omega)$, which can be written without loss of generality as $\xi(t,\omega)=\xi(\theta_t \omega)$; see \cite[Appendix A]{Arnold98} and \cite{csg11}. 
 Recalling that 
 \be
 \underset{t\rightarrow \infty}\lim \frac{z(t,\omega)}{t}=0,
 \ee
 almost surely (see \cite[Lemma 3.1]{CLW14_vol1}), we deduce that, for almost all $\omega$, the positive random 
variable 
  \be
 R(\omega)=\frac{\beta}{\kappa}+ \int_{-\infty}^0 e^{\kappa s}\xi(\theta_ s \omega) \d s,
 \ee
 is well-defined, provided that  $\kappa >0$, which we assume hereafter to be the case.
 
For $\epsilon>0$, the random set 
 \be
 \mathfrak{B}_{\epsilon}(\omega):= \{ v=(x,y,r)\in X: \widetilde{V}(v) \leq
(1+\epsilon)R(\omega)\},
\ee
 in $X:=\mathbb{R}^{d+d'}\times \mathbb{R}^{pq}$ is then, by assumption on $V$, almost surely bounded. One can conclude forthwith that $\mathfrak{B}_{\epsilon}(\omega)$ is a random set that is pullback absorbing for the RDS $\varphi$ generated by \eqref{genMSM_transf}; namely,
for any bounded  set $\mathfrak{B} \subset  X$, there exists a positive random absorption time  $t^\ast(\omega, \mathfrak{B})$ such that, for every $t\geq t^\ast(\omega, \mathfrak{B}),$
\be
\widetilde{V}(\varphi(t,\theta_{-t} \omega)\mathfrak{B}) \subset \mathfrak{B}_{\epsilon}(\omega),
\ee
i.e.~the compact random set $\mathfrak{B}_{\epsilon}$ is absorbing every bounded deterministic {set} $\mathfrak{B}$ of $X.$ Standard results from RDS theory (see \cite[Theorem 3.11]{CF94}) allow us then to conclude that
a global random attractor exists for the RDS generated by \eqref{genMSM_transf} and therefore for the one associated with \eqref{Eq_MSMgen}. 

We have thus proved the following theorem.

\begin{thm}\label{Main_thm}
Let us consider  a system  \eqref{Eq_MSMgen} that generates an RDS.  
Assume that there exists a map $V: \mathbb{R}^{d+d'} \rightarrow
\mathbb{R}^+$, such that $V \in C^1(\mathbb{R}^{d+d'}, \mathbb{R}^+)$,
which has the property that pre-images of bounded sets are bounded and 
that the conditions \eqref{cond1}\textendash\eqref{cond4} hold. Assume, furthermore, that
for all $u\in \mathbb{R}^{d+d'}$ and for all $r\in \mathbb{R}^{pq}$,
\be 
\widetilde{h}(u, r)= h(u, r)+D r,  
\ee
with $D$ a $pq\times pq$  positive-definite matrix, and
that the structural constants involved in \eqref{cond1}--\eqref{cond4} can be chosen such that
\be\label{kappa_cond}
\kappa = \min(\alpha -a -m_1-\epsilon_1 \gamma ;2( \lambda -m_2))>0
\ee
for some $\epsilon_1>0$, 
where $\lambda$ denotes the smallest eigenvalue of $D$.

Then the system  \eqref{Eq_MSMgen}  possesses a global random attractor which is pullback attracting the deterministic bounded sets of $\mathbb{R}^{d+d'}\times \mathbb{R}^{pq}$. 
\end{thm} 

{\bf Interpretation and practical consequences.}  Condition \eqref{cond1} of Theorem \ref{Main_thm} identifies a natural dissipation condition that the self-interactions terms $f_1$ and $f_2$ should satisfy in the absence of coupling between the $x$- and $y$-variables and forcing from the $r$-equation.  Condition \eqref{cond2} and conditions { (\ref{cond3}, \ref{cond4})} indicate what  the cross-interactions between the $x$- and $y$-variables, on the one hand, and the effects 
{of the} forcing by the $r$-variable on the $y$-equation, 
on the other, should satisfy for the system \eqref{Eq_MSMgen} to be pathwise dissipative in a pullback sense \cite{csg11}.  Condition \eqref{kappa_cond} { dictates the appropriate balance between} these combined effects for such dissipativity to occur.

Conditions \eqref{cond1}\textendash\eqref{cond4} 
{may also prove practically} useful in the design of an MSM. For instance, when an appropriate energy $V$ has been identified such that \eqref{cond1} 
{is} satisfied,
the cross-interactions $g_1$ and $g_2$ can be designed according to the constraint
\be\label{Eq_cross_constraint}
  \langle g_1(u), \nabla_{x}  V (u)\rangle + \langle g_2(u), \nabla_{y}  V (u)\rangle=0,
\ee
while conditions  \eqref{cond3} and \eqref{cond4} are 
{easy to satisfy 
when} $V$ is quadratic and $\widetilde{h}$ is linear, { as in \cite{Majda_Harlim2012, MMH}.}

If $V(u)=\frac{1}{2} (\|x\|^2+\|y\|^2)$, it is quite natural  to satisfy \eqref{cond1} simply by seeking $f_i$'s of the form
\be
f_i(x_i)=-A_i x_i +B_i(x_i,x_i)+F_i, \; 
{i = 1,2,}
\ee
with $x_1=x$ and $x_2=y$, 
{while} the linear parts satisfy the dissipation conditions 
\be\label{Eq_dissipA}
\langle A_ix_i, x_i \rangle \geq \nu_i \|x_i\|^2,
\ee
and the bilinear parts satisfy the energy-preserving conditions 
\be\label{Eq_self_energy-preserv}
\langle  B_i(x_i,x_i),x_i \rangle=0, \; 
{i = 1,2.}
\ee

The constraint \eqref{Eq_cross_constraint} 
{becomes therewith}
\be\label{Eq_cross_energy-preserv}
  \langle g_1(x,y),x\rangle + \langle g_2(x,y), y \rangle=0,
\ee
{and, given a linear $\widetilde{h}$ 
and 
bilinear $g_i$'s, 
Eqs.~(\ref{Eq_self_energy-preserv}, \ref{Eq_cross_energy-preserv}) lead} to a class of energy-preserving MSMs considered in \cite{Majda_Harlim2012,MMH}. 
{For this class,} Theorem \ref{Main_thm} ensures the existence of a random attractor provided that the structural constants involved satisfy \eqref{kappa_cond}.  

Clearly, the conditions of Theorem \ref{Main_thm} allow for the design of a { considerably broader class of stable} MSMs. For instance, a closer look at the constraint \eqref{Eq_cross_constraint} allows us to formulate the following corollary.
\begin{corr}\label{Main_corr}
Let us consider a system  \eqref{Eq_MSMgen} for which $d=d'$, and that generates an RDS.  
Assume that there exists a continuously differentiable
map $V: \mathbb{R}^{2d} \rightarrow
\mathbb{R}^+$, 
such that condition \eqref{cond1} is satisfied and 
that pre-images of bounded sets { under $V$} are bounded.

Assume that there exists a 
continuously differentiable, scalar-valued function $G$ on $\mathbb{R}^{2d}$ such that
\bea\label{Eq_antideriv}
& g_1(u)=\nabla_y G(u), \quad
& g_2(u)=-\nabla_x G(u),
\eea
{and that} $G$ is a first integral of the flow associated with the Hamiltonian $V$, i.e. 
{that} the Poisson bracket of $G$ with $V$ 
{vanishes,}  
\be\label{Eq_Poisson}
\{G,V\}=0.
\ee
Then, if conditions { (\ref{cond3}, \ref{cond4})} are satisfied, the conclusion of Theorem \ref{Main_thm} holds.
\end{corr}

\begin{proof}
The proof of this corollary boils down to 
expanding \eqref{Eq_Poisson} 
in canonical coordinates 
to yield 
\cite[p.~215]{arnol1989mathematical}
\be
\sum_{i=1}^{d} \frac{\partial G}{\partial y_i}  \frac{\partial V}{\partial x_i} - \frac{\partial G}{\partial x_i}  \frac{\partial V}{\partial y_i}=0, 
\ee
which, { by using \eqref{Eq_antideriv},} is nothing else than \eqref{Eq_cross_constraint}. \end{proof}

This corollary shows that, once the self-interactions terms $f_1$ and $f_2$ have been designed to be dissipative, according to \eqref{cond1}, with respect to some energy $V$, the cross-interactions between the main level variable $x$ and the 
next-level hidden variable $y$ can be derived from any constant of motion $G$ of the flow generated by the auxiliary Hamiltonian system 
\bea
& \dot{p}=\nabla_q V(p,q), \quad
& \dot{q}=-\nabla_p V(p,q).
\eea
The determination of such cross-interactions can thus benefit 
from efficient symplectic integrators techniques {\cite{Varadi2003, hairer2006geometric}} that we leave for future research.

 Whatever the level of generality of the MSM formulation, though, the selection of the corresponding allowable
 self- and cross-interactions should be balanced with the constraint of reduction of  the unexplained variance contained in the residuals of the main and subsequent layers, when compared to an MSM written in its original EMR formulation \eqref{Eq_MSM}. This variance-based constraint is also to  
 be balanced against the 
correlation-based criterion formulated in Section \ref{Practical_MZ} below.  The latter criterion relies on a reformulation  of an MSM into a system of stochastic integro-differential equations that is described below,  
a reformulation that allows for a comparison with the optimal closure model 
 one can obtain from a time series of partial observations of a large system; see Section \ref{MSM-MZ}. 

\br\label{rem:EnKF}
{In practice, an ensemble Kalman filter} 
can be used to learn MSMs subject to constraints such as { (\ref{Eq_self_energy-preserv}, \ref{Eq_cross_energy-preserv})} \cite{Majda_Harlim2012,MMH}. 
{In \ref{App_energy}, we describe an alternative approach, 
in which} such constraints can be naturally 
{incorporated into the recursive} minimization procedure described in Section \ref{sec:emr_form}.  
\er

\subsection{MSMs as systems of stochastic integro-differential equations}\label{Sec_MSM_is_integro-stoch}
In this section we 
{show how to rewrite an MSM} as a system of stochastic integro-differential equations  
{like those that arise in the MZ} 
formalism. 
{We restrict ourselves here} to the case 
{of $h$ in \eqref{Eq_decomp_h} being} given by
\bea\label{Eq_cond_MSMisMZ}
& h(u,r)=-D  r +\tilde{h}(x),\\
& g_2(u)=g_2(x),
\eea
i.e. 
{we drop the} dependence 
{of $\tilde{h}$ (respectively~of $g_2$) on $ r$ and $y$ (respectively~on $y$), 
while $D$ is}
still a $pq\times pq$ positive-definite matrix. 

In this case, the integration of the last equation in \eqref{Eq_MSMgen}  gives, for $ r(0)=0$, 
\be\label{Eq_stoch_convol}
 r(t)=\int_{0}^t e^{-(t-s)D} \tilde{h}(x(s))\d s +\int_{0}^t e^{-(t-s)D} \Sigma\;\d W_s,
\ee
where the last integral is understood as a {\it stochastic convolution} in the sense of It$\overline{\mbox{o}}$ \cite[Chap. 5]{DPZ08}.
Denoting by $S(t)$ the flow associated with  $\dot{y}=f_2(y)$, we obtain 
{likewise, for $y_0=0$,}
\be\label{Eq_convol_repre_y}
y(t)=\int_0^t S(t-s)\Big(g_2(x(s))+\Pi r(s)\Big) \d s, 
\ee
with $r(s)$ given by \eqref{Eq_stoch_convol}.  

Let us assume, furthermore, that
\be\label{Eq_g1}
g_1(u)=\beta(x,y)+C y,
\ee
where $C$ is
a $d\times d'$ matrix, 
{while $\beta(x,y)$} does not contain linear terms in the $y$-variable nor homogeneous terms in the $x$-variable.
{Given} these assumptions, 
{we can use Eqs.~(\ref{Eq_stoch_convol}, \ref{Eq_convol_repre_y}) to rewrite} the system of SDEs { \eqref{Eq_MSMgen} 
as the following {\it RDE with retarded arguments}:} 
\be\label{Eq_MSMfunc}
\frac{\d x}{\d t}=f_1(x)+ \mathbb{G}(t,x_t) +\xi (t,\omega),
\ee
with
\be\label{Eq_shift}
x_t(s):=x(t+s), \; -t\leq s\leq 0,
\ee
\be\label{Eq_repeated_convol}
\xi(t,\omega)=C \int_0^t S(t-s')\Big (\Pi \int_{0}^{s'} e^{-(s'-s)D} \Sigma\;\d W_s (\omega) )\Big )\d s'.
\ee

{Here} $\mathbb{G}(t,x_t)$ accounts for (i) the terms {  in \eqref{Eq_g1} that come} from $\beta(x(t),y(t))$, 
where $y(t)$ has been replaced by its expression \eqref{Eq_convol_repre_y}; and (ii) the terms 
$C \big(\Pi \int_{0}^t e^{-(t-s)D} \tilde{h}(x(s))\d s +\int_0^t S(t-s)g_2(x(s))\d s\big)$.  Note that each of these terms 
{is given by the} following type of integral:
\be
\int_0^t F(t,s,x(s))\d s;
\ee
{the latter can be rewritten, by using \eqref{Eq_shift},} as follows: 
\be
\int_0^t F(t,s,x(s))\d s=\int_{-t}^0 F(t,t+s,x (t+s)) \d s=\int_{-t}^0 F(t,t+s,x_t(s)) \d s:=\mathbb{F}(t,x_t),
\ee
{which explains} the functional dependence in \eqref{Eq_MSMfunc} on the past history $x_t$ of $x$.

The abstract reformulation of \eqref{Eq_MSMgen}  
{as a closed equation \eqref{Eq_MSMfunc} that 
describes} the evolution of $x$ 
shows that an MSM  written in a closed form involves a functional dependence on the time history of the observed variables. 
{This important characteristic of a time-continuous EMR} was already 
{highlighted, in a simpler context,} in the Supporting Information of \cite{CKG11}.

 Note that $\xi(t,\omega)$ has a straightforward interpretation when 
{taking into} account the multilayer structure of an MSM. Indeed, if we remember that the $r$-equation in  \eqref{Eq_MSMgen} 
{has} $p$ levels and if we assume, furthermore, that the matrix $D$ is block-diagonal and given by
 \be\label{Eq_Ddiag}
 D=\mbox{diag}(D_1,\cdots,D_p), 
 \ee 
 where each $D_j$, $j = 1,\ldots,p$, is 
a $q \times q$ matrix, 
 it follows that $\xi(t,\omega)$ is obtained, due to the assumption on $\Sigma$, as the following convolution
 \be
  \xi(t,\omega)=C \int_0^t S(t-s)\Pi z_1(s) \d s;
 \ee
{ here} the stochastic process $z_1$ is obtained from  the successive integrations  of the following equations:
\beas\label{Eq_z1}
\d z_p & = -D_p z_p \d t + Q\,\d W,\\
\d z_{p-1} & = -D_{p-1}  z_{p-1} \d t +z_p \d t,\\
\d z_{p-2} & = -D_{p-2}  z_{p-2} \d t +z_{p-1} \d t,\\
& \vdots \hspace{10ex} \vdots\\
\d z_{1}=&-D_{1}   z_{1} \d t +z_2 \d t.
\eeas
In other words, the process $z_1$ can be rewritten in terms of repeated convolutions as follows  
\be\label{xi_t_expanded}
z_1(t)=\kappa_1 \ast \cdots \ast  \kappa_{p-1}\ast z_p(t),
\ee
where the kernels $\kappa_k$ are given by the following exponential matrices; 
\be
\kappa_k(t) = e^{t D_k}, \; 1 \leq k \leq p,
\ee
and the $\kappa_k$ will also be
called memory kernels in the sequel.

In more intuitive terms,  $z_1$ is a stochastic process { that results} from the  propagation of a red noise through the successive linear MSM layers, up to the level just preceding the main level, namely the $y$-equation in \eqref{Eq_MSMgen}. Mathematically, the power spectrum of $z_1$ is red  and its statistics are Gaussian. The interest of system \eqref{Eq_z1} is that it facilitates the simulation of such a ``reddish" noise. This reddish noise is then convoluted with the possibly nonlinear flow associated with $\dot{y}=f_2(y)$ to give rise to the stochastic process $\xi(t,\omega)$, which  may therewith be non-Gaussian.

We  can now summarize the above analysis in the following proposition.
\begin{prop}\label{Main_Prop}
Under the assumptions  of Theorem \ref{Main_thm}, and if the conditions \eqref{Eq_cond_MSMisMZ}, \eqref{Eq_g1} and \eqref{Eq_Ddiag} are satisfied, then any solution of  \eqref{Eq_MSMgen} emanating from $(x_0,0,0)$ satisfies the following RDE with retarded arguments, 
\be\label{Eq_MSMfunc2}\tag{{{\it RDDE}}}
\boxed{
\frac{\d x}{\d t}=\overbrace{f_1(x)}^{(a)}+ \overbrace{\mathbb{G}(t,x_t)}^{(b)} +\overbrace{\xi (t,\omega)}^{(c)},}
\ee
where 
 \be\label{Eq_xi_interp}
  \xi(t,\omega)=C \int_0^t S(t-s)\Pi z_1(s) \d s,
 \ee
$z_1$ solves \eqref{Eq_z1}, and $S(t)$ 
is the flow generated by the vector field $f_2$;
moreover,
\be
 \mathbb{G}(t,x_t)=\int_0^t S(t-s)g_2(x(s))\d s+C \Pi \int_{0}^t e^{-(t-s)D} \tilde{h}(x(s))\d s +\beta(x(t),y(t)),
\ee
with $y(t)$ given by \eqref{Eq_convol_repre_y}.
\end{prop}

Equation~\eqref{Eq_MSMfunc2} shows then that an MSM --- with cross-interactions between the observed and hidden variables subject to the conditions of Proposition \ref{Main_Prop} --- 
decomposes the dynamics of the observed variables into (a) nonlinear self-interactions, embedded in (c) a reddish background, plus (b) {\it state-dependent correction terms} represented by convolution integrals. The reddish background (c) stands for the cross-interactions and self-interactions that are not accounted for, respectively, by the convolution terms in (b) and by the nonlinear terms in (a). This noise term also accounts for the lack of knowledge of the full initial state due to partial observations. Stated otherwise, the deterministic terms in (a) provide a Markovian contribution to $\d x/\d t,$ while the terms in (b) constitute a non-Markovian contribution  {\it via} the past history of $x$, as $t$ evolves, and the (c)-term represents the fluctuations that are not modeled by the terms in (a) and (b);  the (c) term also accounts for the uncertainty in the knowledge of the full initial state, due to partial observations \cite{Chorin_Hald-book}. 

We have thus clarified that any system \eqref{Eq_MSMgen}, 
subject to conditions like those of Proposition \ref{Main_Prop}, 
resembles the closure models derived by applying the MZ 
formalism; see \cite[Eq.~(6)]{Chorin_MZ} and \cite{Chorin_MZ, Chorin_al06, Chorin_Stinis06} for further details.  A discrete formulation of an MSM --- such as proposed by the original formulation of an EMR {in \cite{kkg05, kkg09_rev} ---}
presents, nevertheless, substantial advantages in practice, since such a discrete system is in general much easier to integrate numerically than the system of stochastic integro-differential equations \eqref{Eq_MSMfunc2}.  This is particularly true in instances where the memory kernel $\mu_k$ in \eqref{xi_t_expanded} decay slowly, in the sense that the observed variables $x$ evolve on a time scale comparable to that of the decorrelation time of such kernels.\footnote{In such a case, the direct numerical integration of \eqref{Eq_MSMfunc2} becomes prohibitive, due to the cost of evaluating the integral terms. Such numerical difficulties have  limited, so far, the application of the MZ formalism in the derivation of efficient closure models of partially observed  high-dimensional systems for which the determination of 
{$\mathbb{G}$} becomes a non-trivial task in the case of the so-called  intermediate-range memory effects; see \cite{Stinis06}. Nevertheless, in  the case of long-range memory effects, {the MZ framework was successfully applied 
to} the Euler or Burgers equations {by deriving a so-called $t$-model} \cite{hald_stinis, stinis_euler}.} Such a situation is expected to occur when the separation of time scales between the observed and unobserved variables  is not as sharp as required, for instance, in applying stochastic homogenization techniques 
\cite{Majda2001}; see however the recent works   \cite{MS11,GM13} for milder assumptions.

The observation that an MSM can be recast into a system of stochastic integro-differential equations that resembles the 
{ GLE of the MZ} 
formalism raises a natural question with respect to our closure problem ($\mathfrak{P}$) formulated in Section \ref{sec:emr_form}, {namely to} which extent {does} an MSM 
provide {a good} 
approximation of the GLE predicted by the 
{MZ} formalism? The next section sets up the mathematical framework to address this problem, 
followed by the formulation of a simple correlation-based criterion 
in  Sec.~\ref{Practical_MZ} to 
{solve it.} 

Our criterion is based on Proposition \ref{Main_Prop} and, in particular, on the representation of $\xi(t,\omega)$ provided by \eqref{Eq_xi_interp};
 it helps provide information on the degree of approximation of the GLE by a given MSM, in terms of { the correlation between the residual noise term (c) in the \eqref{Eq_MSMfunc2} and the observed variables $x(t)$.} 
Such a criterion will turn out to be useful in applications, given the fact that the GLE constitutes the optimal closure model that can be achieved from an infinite time series of partial observations, as pointed out by Lemma  \ref{Main_lemma} below.


\section{Generalized Langevin equation (GLE) for optimal closure from a time series}\label{MSM-MZ}

In this section, we assume that the scalar field $u(t_k, \xi_j)$ in $d$ dimensions, as expanded in \eqref{equ:expand}, 
is the spatial coarse-graining of a discrete field  $\mathcal{U}_{ \mathcal{I}}:=\{u(t_k,\xi_j)\}_{j\in \mathcal{I}}$ in $n$ dimensions.
Typically, card($\mathcal{I}$)=$n \gg d$ and an expansion similar to \eqref{equ:expand} holds in $n$ dimensions, that is:  
\be
u(t_{k},\xi_j)=\sum_{i=1}^{n} y_i(t_k) E_i(\xi_j),\quad j\in { \mathcal{I},} \; 
{ k = 1, \ldots ,N.}
\ee
We also assume 
that the evolution of  $y$ is governed by the { ODE} system 
\be\label{Eq_main}
\dot{y}=\R(y).
\ee

The goal of this section is to show how the  
{GLE from the MZ} formalism can be 
{theoretically} derived from a time series. { This mathematical derivation lays} the foundations for seeking the Markovian part of an MSM by regression methods in practice, { as discussed and applied in Secs.~\ref{Practical_MZ}--\ref{sec_pop_dyn}.}

To simplify the presentation, we assume that the vector field $\R$ is  
{\it continuously differentiable} on $Y:=\mathbb{R}^{n}$, endowed with the basis { $\mathcal{B}' = \{E_i: i = 1, \ldots, n\},$} such that the flow $\{T_t\}$ associated with it is well-defined for all $t$ and possesses 
a compact global attractor $\mathcal{A} \subset Y$ 
\cite{Tem97}. 
We 
assume, furthermore,  
that $\{T_t\}$ possesses an invariant probability measure $\mu$, which is {\it physically relevant} \cite{eckmann_ruelle, csg11}, {\it i.e.}:
\be\label{Eq_phys_rev}
\underset{T\rightarrow \infty}\lim \frac{1}{T} \int_0^{T}  \varphi(T_t( y)) \d t =\int_{\mathcal{A}} \varphi( y) \d \mu ( y),
\ee
for almost all  $y \in Y$ (in the sense of Lebesgue measure) and for any observable $\varphi \in L^1_{\mu}(Y).$  Recall that, 
like all measures invariant under $T_t,$ an invariant measure that satisfies \eqref{Eq_phys_rev} is supported by the global attractor $\mathcal{A}$; see, for instance, \cite[Lemma 5.1]{chekroun_glatt-holtz}.

We briefly outline now how 
the MZ formalism helps derive
an abstract {\it closure model} 
for the description of  the dynamics of $x=P y$, where $P$ is the projection from $Y$ onto $X:=\mathbb{R}^d$ and the latter is endowed with the basis { $\mathcal{B} = \{E_i: i = 1, \ldots, n\},$} $d \ll n$.  By a  closure model we mean a model that describes the evolution of $x(t)=Py(t)$ as a function of the $x$-variable only, plus some possible 
forcing terms that model the information lost by  
applying $P$,  {\it i.e.}, due to the partial character of the available observations. 
Our approach thus follows 
\cite{Chorin_MZ, Chorin_al06}, who rely on the use of conditional expectation with respect to a meaningful invariant measure such as the $\mu$ introduced above. 
This framework will allow us, 
 further below, to prove the main result of this section, formulated as Lemma \ref{Main_lemma}. 
Practical consequences of this Lemma for MSMs are discussed 
 in Section \ref{Practical_MZ} below.

For convenience, we denote by $v: Y\rightarrow X$  the projection $P$ so that, in particular, 
\bes
v_i( y):=(P y)_i\,, \;\;\; i = 1, \ldots ,d.
\ees
By differentiating  $v_i(T_t( y))$ with respect to time, we obtain
\beas
\frac{\partial }{\partial t} v_i (T_t( y)) & = \nabla_{ y} (v_i(T_t  y)) \cdot \frac{\d T_t( y)}{\d t} \\
       &  = \nabla_{ y} (v_i(T_t  y))\cdot \R (T_t( y)) = \Ld v_i(T_t( y));
 \eeas
here $\Ld$ is the Lie derivative, acting on continuously differentiable functions $h$, along the vector field $\R$ given by:
\be\label{Eq_L}
\Ld h ( y):=  \nabla h( y) \cdot \R( y),
\ee
Note that, by definition, $v_i(T_t( y_0))$ is 
just the $i^{\mathrm{th}}$-component $x_i(t)$ of $x(t; y_0)$, namely the projection onto $X$ of the solution of \eqref{Eq_main} emanating from $ y_0 \in Y$.

By introducing --- at this point only formally --- the time-dependent family of {\it Koopman operators}\footnote{For further details about Koopman semigroups and operators, see \cite{budivsic2012applied,Chekroun_roux, cornfeld1982ergodic,lasota1994chaos} and references therein. For this article, it suffices to note that $U_t$ describes the action of the flow $T_t$ on observables $u:Y\rightarrow X$.}: 
\be\label{Def_Koop}
(U_t u)( y_0):=u(T_t  y_0), \;  y_0 \in Y,
\ee
defined for observables $u:Y\rightarrow X$ { that live} in some appropriate functional space\footnote{Such a space could  be chosen to be for instance $D_p=\{ u\in L^{p}_{\mu} (Y;X) \; | \; Au:=\lim_{t\rightarrow 0} \ t^{-1}(U_t u -u) \textrm{ exists}\}$ for some $ p \in [1,\infty]$, where the limit is taken  in the sense of strong convergence \cite{brezis_book,EN00}.  We do not consider in this article the delicate question of the choice of functional spaces 
{ that} characterize the mixing properties of the flow $T_t$. Such considerations require typically 
spaces that take into account the stable and unstable manifolds of the attractor, when the latter supports a { Sinai}-Bowen-Ruelle measure; see, for instance, \cite{butterley2007smooth} in the case of Anosov flows. { For a rigorous treatment in the context of Hamiltonian systems, see \cite{GKS04}.}

},
the evolution of $x(t; y_0)=v(T_t  y_0)$ 
{ is} governed by
the following system of linear, hyperbolic PDEs { with non-constant coefficients in $n+1$} variables:
\be\label{Eq_Liouv}
\frac{\partial }{\partial t} U_t v( y_0)=\mathfrak{L} (U_t v)( y_0), \; y_0 \in Y,
\ee
where $\mathfrak{L} \Psi:=(\mathcal{L} \Psi_1,\dots,\mathcal{L} \Psi_n),$ with $\mathcal{L}$ given by \eqref{Eq_L}.
Within the appropriate functional setting,  it can be shown that $\{U_t\}_{t\geq 0}$ forms a genuine semigroup, {\it i.e.} it describes not only the dynamics of $v$ but of other observables as well, and its generator is given by $\mathfrak{L}$, {\it i.e.}  $U_t=e^{t\mathfrak{L}}$. 
{In} other words, $\mathfrak{L}$ can be interpreted as the ``rate of change" of  $U_t$; see {\it e.g.} \cite[Chap. 7.6]{lasota1994chaos}, or \cite{butterley2007smooth} for a more advanced treatment in the { dual} case of Perron-Frobenius semigroups.

The Liouville-type equation \eqref{Eq_Liouv} can be rewritten in a more  convenient form for our purposes; in particular, we will see that the existence of an invariant measure $\mu$ satisfying \eqref{Eq_phys_rev}  plays an essential role in connecting the MSMs introduced in the previous section with the 
{MZ} formalism.  Let $Z$ be the complement of $X$ in $Y$, i.e.  
\be
Y=X\oplus Z,  
\ee
and let $\Psi : Y\rightarrow X$ be a continuous function.
The decomposition of $Y$ allows us to split any  $ y\in Y$  as the sum $x+z,$ with $x\in X$ and $z \in Z$; { here both $x$ and $ z$ are} uniquely determined by the projection $v: Y\rightarrow X$ and its complementary projection, $\mathrm{Id}_Y -v$, respectively.

The existence of  an invariant measure $\mu$ allows us next to define the {\it conditional expectation}\footnote{Note that $\int_{Z} \Psi( y) \d \mu_{x}(z)$ { in Eq.~\eqref{Def_cond_expec}}
is finite by assuming { that} $\Psi$ continuous, since the support of the invariant measure $\mu$, supp$(\mathcal{\mu})$ is compact. Actually, by the Fubini theorem \cite{Rudin}, it suffices to assume that $\Psi \in L^1_{\mu}(Y;X)$ for this integral to be well defined.} 
{of $\Psi$ for each $v(x$),}
\be\label{Def_cond_expec}
\mathbb{E}[\Psi | v] (x):=\int_{Z} \Psi( y) \d \mu_{x}(z). 
\ee
{Here} $\mu_{x}$ is the probability measure on { the unobserved factor space} $Z$ obtained by disintegration of $\mu$ above $x$, {\it i.e.}, for all Borel sets $B$ and $F$ of $Z$ and $X$, respectively, we have that:
\be\label{Eq_disint}
\mu(B\times F)=\int_{F} \mu_{x} (B) \d \mathfrak{m} (x),
\ee
where $\mathfrak{m}$ is the {\it push-forward} of the measure $\mu$ by $v$, {\it i.e.}  $\mathfrak{m}(F)=\mu(v^{-1}(F)),$ for any Borel set $F$ of $X$.

The existence of $\mu_{x}$ such that \eqref{Eq_disint} holds is ensured by the {\it disintegration theorem of 
probability measures}; see for instance \cite[Section 10.2, pp. 341--351]{dudley_book} or \cite[Chapter 5]{kallenberg_book}\footnote{See also \cite[\S 4]{CF94} and \cite{Arnold98,csg11,Crauel02} for disintegration of probability measures arising in { RDS} theory;
in this theory, 
the probability measure $\mathfrak{m}$ on the right-hand side of \eqref{Eq_disint} is typically replaced by the probability measure $\mathbb{P}$ associated with
the driving system.}. The probability measure $\mu_{x}$ can be interpreted as reflecting the statistics of the unobserved variables $ z$ when $x$ has been observed; see \cite{Chek_al13_RP} for further details.

\begin{rem}\label{Rem_ave}
From the definitions of $\mathfrak{L}$ and $v$, we see  that $\mathfrak{L}v$ corresponds  to the vector field 
\bea
v\circ \R: \; & X\rightarrow X,\\
& y\mapsto(R_1( y),....,R_d( y)),
 \eea
where $\R$ is the  vector field governing the evolution of the full system \eqref{Eq_main}. Thus  $\mathfrak{L}v$ is 
{ but} the first $d$ components of the vector field $\R$ and \eqref{Eq_Liouv} corresponds then to the functional formulation of
\be\label{Eq_MZ0}
\frac{\d x}{\d t} =v(\R(x +z)),  \; x \in X,\;  z\in Z,
\ee
where $Y=X\oplus Z$.

The truncated vector field $v \circ \R$  is still a function of $ y$ and, in particular, it depends on the unobserved variables $ z$.  It is then reasonable to seek the vector field in $X$  that best approximates $\mathfrak{L}v$ in a least-square sense, {as weighted by the invariant measure $\mu$, i.e. in $L^2_{\mu}(Y;X)$.} In other words, we seek an $X$-valued function of the
{observed, $d$-dimensional} $x$ only,  which best approximates the $X$-valued function $\mathfrak{L}v$ of the { full, $n$-dimensional} $y$. It is exactly this approximation that provides the conditional expectation $\mathbb{E}[\mathfrak{L} v| v]$ corresponding 
to the vector field
\be\label{Eq_reduced_vec_field}
\overline{v\circ \R}(x):=\int_{Z} v(\R(x+z)) \d \mu_{x} (z), \; x \in X,\;  z\in Z.
\ee
The averaging with respect to the unobserved variables $ z$ that occurs in Eqs.~\eqref{Def_cond_expec} and \eqref{Eq_reduced_vec_field} becomes therewith intuitively clear.
\end{rem}

Recalling now the stated purpose of deriving a 
closure model for $x(t)$, Remark \ref{Rem_ave} above leads naturally to decompose $v\circ \R$ into its averaged part, given by \eqref{Eq_reduced_vec_field},
and a fluctuating part, 
\be\label{Eq_decompR}
v(\R(x +z))=\overline{v\circ \R}(x) +(v(\R(x +z))-\overline{v\circ \R}(x)).
\ee

The parametrization of the fluctuating part $v(\R(x +z))-\overline{v\circ \R}(x)$ as a function of the $x$-variable is at the core of the MZ-formalism.  This task is achieved through the perturbation theory of semigroups \cite{EN00}. 
Leading up to this task, let us first note that, since $\mathfrak{L}$ and $U_t$ commute, $\mathfrak{L} U_t=U_t \mathfrak{L} $ \cite{Chorin_MZ}, we can  rewrite \eqref{Eq_Liouv}
as 
\begin{equation*}
\frac{\partial}{\partial t} U_t v= U_t \mathfrak{L} v; 
\end{equation*}
{ this, in turn, can be} rewritten as:
\be\label{Eq_MZ1}
\frac{\partial }{\partial t} U_t v = 
U_t \mathbb{E}[\mathfrak{L} v| v]+U_t ( \mathfrak{L} v- \mathbb{E}[\mathfrak{L} v| v]).
\ee

Equation~\eqref{Eq_MZ1} is the functional formulation of \eqref{Eq_MZ0}, once \eqref{Eq_decompR} has been applied. The interest of this formulation is that 
the {\it variation-of-constants formula}, in its Miyadera-Voigt form \cite[{ Section 3c}]{EN00}\footnote{ 
The Miyadera-Voigt form of the variation-of-constants formula is also known as the Dyson formula in the physics literature \cite{Chorin_MZ}.},
yields a decomposition of  $U_t$ into two terms; as we will see, these two terms
have useful interpretations in statistical mechanics \cite{Mor65,nakaji58,PR61,Zwa60,Zwa64}.

The decomposition of $U_t$ is achieved by considering the 
operator $B$, acting on appropriate observables $\Psi:Y\rightarrow X$,
\be\label{Eq_B}
B\Psi :=\mathbb{E}[\mathfrak{L}\Psi | v],
\ee 
as a perturbation of the operator $\mathfrak{L}$, 
and writing
\be
A\Psi= \mathfrak{L} \Psi-  \mathbb{E}[\mathfrak{L}\Psi  | v].
\ee

Using the notation $V_t$ for the semigroup generated by $A$, 
the above-mentioned variation-of-constants formula 
leads (formally) to:
\be\label{Eq_MZ_decomp}
U_t \Psi =V_t \Psi+\int_{0}^t U_{t-s} B V_s \Psi \d s.
\ee

It follows that the evolution of an observable $\Psi:Y\rightarrow X$ under $V_s$  in \eqref{Eq_MZ_decomp} is described by the following 
{ PDE in} $n+1$ variables: 
\be\label{Eq_MZ_ortho}
\frac{\partial }{\partial t} V_t \Psi=A\Psi = \mathfrak{L}V_t\Psi -\mathbb{E}[\mathfrak{L} V_t \Psi | v].
\ee

\begin{rem}\label{Rem_ortho}
Note that  
\be
\frac{\partial}{ \partial t} \mathbb{E}[V_t \Psi | v] = \mathbb{E}[\frac{\partial }{\partial t}(V_t \Psi) | v] = \mathbb{E}\Big[\mathfrak{L}  V_t\Psi -\mathbb{E}\big[\mathfrak{L}  V_t \Psi | v\big]\Big| v\Big]=0, 
\ee
so that, if $\Psi$ is orthogonal to the space spanned by functions of $X$ = Im($v$) alone, then $\mathbb{E}[V_t \Psi| v]=0$ for all $t$. For { this} reason, equation \eqref{Eq_MZ_ortho} is known as providing the {\it orthogonal dynamics} in the MZ terminology.
\end{rem}

According to this remark, by taking 
\be
\Psi=\mathfrak{L} v - \mathbb{E}[\mathfrak{L} v| v],
\ee
we have thus that $V_s \Psi $ in \eqref{Eq_MZ_decomp} evolves in the subspace that is orthogonal to  the space spanned by functions of $X$. The term $B V_s \Psi$  corresponds to the average vector field on $X$ given by $\overline{(V_s \Psi) \circ \R}$, where the average is taken over $Z$ as in  \eqref{Eq_reduced_vec_field}, so that $\int_{0}^t U_{t-s} B V_s \Psi \d s$ is a function that depends on $x\in X$ only. If $x$ is given by the time-dependent function $v(T_t   y_0)$, then $\int_{0}^t U_{t-s} B V_s \Psi \d s$ is a function of the past values of $x$,  
i.e., a  memory term.

By applying \eqref{Eq_MZ_decomp}  with $\Psi=\mathfrak{L} v - \mathbb{E}[\mathfrak{L} v| v]$ and using the expression of $B$ given in \eqref{Eq_B}, 
the term $U_t ( [\mathfrak{L} v- \mathbb{E}[\mathfrak{L}  v| v])$ in \eqref{Eq_MZ1}
may be 
rewritten accordingly, and we arrive thus at the 
following {\it generalized Langevin equation (GLE)} :
\be\label{Eq_MZ_GLE}\tag{{{\it GLE}}}
\boxed{
\frac{\partial }{\partial t} U_t v = \mathcal{R}(U_t v )+\int_{0}^t U_{t-s} \mathcal{G}(v;\eta_s)\d s +\eta_t.}
\ee
Here 
\be
 \mathcal{R}(v)=\mathbb{E}[\mathfrak{L} v| v], \quad 
\mathcal{G}(v;\eta_s)=\mathbb{E}\big[\mathfrak{L}\eta_s \big| v\big],
\ee
and
\be 
\eta_t=V_t (\mathfrak{L} v - \mathbb{E}[\mathfrak{L} v| v]);
\ee
see \cite{Chorin_MZ} for further details. As explained above, the integral term in \eqref{Eq_MZ_GLE} constitutes the non-Markovian contribution to the evolution of $U_t v$, and thus of $v(T_t   y_0)$; this contribution depends on $U_s $ for  the past interval $0 \leq s \leq t$.

The term $\eta_t$ is a source of fluctuations related to the partial knowledge $x_0$ of the full initial state $y_0=x_0+z_0$.
Recall that the conditional expectation $\mathbb{E}[\Psi | v]$, as defined in  \eqref{Def_cond_expec}, of a function $\Psi$ in $L^2_{\mu}(Y;X)$ can be seen as  the orthogonal projection $\Ps$  onto the space of $X$-valued functions that depend only on the $x$-variable; see Remark \ref{rem_proj}  below. Thus, by introducing the complementary projector 
\be
\Qs=(\mathrm{Id}_{L^2_{\mu}(Y;X)} -\Ps), 
\ee
we obtain from \eqref{Eq_MZ_ortho} that $\eta_t$ solves the following initial value problem:
\be\label{Eq_MZ_ortho2} \left\{
\begin{array}{l@{}} 
{\displaystyle {\frac{\partial }{\partial t} \eta_t} = \Qs\mathfrak{L} \eta_t \quad \text{in }Y,} \\
{\displaystyle {\eta_0 (x_0,z_0)=v\circ \R(x_0+z_0)-\overline{v \circ \R} (x_0), \quad x_0 \in X, \; z_0 \in Z.}}
\end{array}
\right. \ee 
In other words, $\eta_t$ gives the evolution in time, according to the {\it orthogonal dynamics}, of a 
deviation $\eta_0 (x_0, z_0):=v\circ \R(x_0+z_0)-\overline{v \circ \R} (x_0)$ at time $t = 0$, with respect to the conditional expectation \eqref{Eq_reduced_vec_field}.

Since $z_0$ is distributed according to the disintegrated measure
$\mu_{x_0}$,  the fluctuating term $\eta_t$ can be then  interpreted as an  $X$-valued random variable. Since $\mathbb{E}[\eta_0 | v]=0$, we deduce from Remark \ref{Rem_ortho} that $\mathbb{E}[\eta_t | v]=0$ for all $t>0$, {\it i.e.} the conditional expectation of $\eta_t$ remains zero as $t$ evolves, so that $\eta_t$ is uncorrelated with any function of $v$ \cite{Chorin_Hald-book, chorin2000optimal}.

\begin{rem}\label{rem_proj}
 Recall that we may regard $\mathbb{E}[\Psi | v]$ as the orthogonal projection  of $\Psi$ belonging  to 
\be
L^2_{\mu}(Y; X):=\{g:Y\rightarrow X, \,\, \mathrm{measurable \, and \, such \, that} \;  \int_{Y} \| g( y)\|^2 \d \mu ( y) <\infty \},
\ee
onto the space of functions of $X$, since for 
all function $f: X\rightarrow X$ such that $f\circ v \in L^2_{\mu}(Y;X)$,
\be\label{Eq_best_app}
\mathbb{E}[\|\Psi-\mathbb{E}[\Psi| v]\|^2] \leq \mathbb{E} [\| \Psi-f \circ v\|^2],
\ee
where the expectation $\mathbb{E} (\Phi)$ is taken here with respect to $\mu$, that is: 
\be\label{Def_expec}
\mathbb{E} (\Phi)=\int_{\mathcal{A}} \Phi( y) \d \mu ( y), \; \Phi \in L^1_{\mu}(Y).  
\ee
\end{rem}

Equations \eqref{Eq_MZ_GLE} are fewer in number than in \eqref{Eq_main}, i.e.~$d \ll n$, but this advantage is outweighed
by the need to find the fluctuating term $\eta_t$ as a solution of the orthogonal dynamics \eqref{Eq_MZ_ortho2},  along with its 
requisite statistical properties, in order to simulate \eqref{Eq_MZ_GLE} accordingly. What equation \eqref{Eq_MZ_GLE} does provide is a theoretical ``master equation," which can serve as a basis for the design of various practical closure strategies.  

With this purpose in mind, the following fundamental lemma 
shows that actually the 
{reduced, Markovian} vector field $\mathcal{R}(v)$ in \eqref{Eq_MZ_GLE}  can, in principle, be approximated from a time series, when the latter represents partial observations drawn from a physical invariant measure.

\begin{lem}\label{Main_lemma}
Assume that the main system \eqref{Eq_main} possesses an invariant measure $\mu$ that satisfies the ``physicality condition'' of \eqref{Eq_phys_rev}. Let  $v: Y\rightarrow X$ denote the projection onto $X=\mathrm{span}\{E_1,...,E_d\}$, and let $\mathcal{E}(v,\mu)$ be the closure in $L^2_{\mu}(Y;X)$ of the set of functions $f:X\rightarrow X$ such that $f\circ v \in L^2_{\mu}(Y;X).$ 

Then, for almost all initial data $ y_0 \in Y$, 
{in the sense of Lebesgue measure on $Y$,}
\be\label{Eq_key}
\boxed{
\underset{f \in \mathcal{E}(v,\mu)}{\mathrm{argmin}} \Big(\underset{T\rightarrow \infty}\lim \frac{1}{T} \int_0^{T}\Big \| \frac{ \mathrm{d} x}{ \mathrm{d} t}-f(x(t; y_0)) \Big \|^2 \mathrm{d} t \Big)= \mathbb{E}[\mathfrak{L} v |v]
}
\ee
 holds,
where $\mathbb{E}[\mathfrak{L} v |v]$ is the conditional expectation  defined in  \eqref{Def_cond_expec},  
{ while} $x(t; y_0):=v(  y(t; y_0))$ with $ y(t; y_0)$ denoting the solution $ y(t; y_0)$  of \eqref{Eq_main} emanating from $ y_0$.
\end{lem}

\begin{proof} The proof  uses the two facts that (i) $\mu$ is a physical invariant measure in the sense of \eqref{Eq_phys_rev}, and that (ii) $\mathbb{E}[\Ld v |v ] $ is the projection, in $L^2_{\mu}(Y; X)$, of the rate of change of  $U_t v$ onto $X$.
Let $\varphi:Y\rightarrow \mathbb{R}$ denote the observable defined by $\varphi( y)= \| \mathfrak{L} v ( y)- f(v( y)) \|^2$.  

Note that {the functional} $\varphi$ thus defined lives in $L^1_{\mu}(\mathcal{A})$. To see this, first note that --- from \eqref{Eq_L} and the definition  of $\mathfrak{L}$ --- it is not difficult  to show that there exists $C>0$ such that
\be
\| \mathfrak{L} v ( y)\|\leq C \|\R( y)\|, \; 
\text{for all}\, y \in \mathcal{A}.
\ee
This bound, in turn, implies  that  $\mathfrak{L} v$ lives in $L^2_{\mu} (Y;X)$ {and,} since  $\R$ is continuous, the global attractor $\mathcal{A}$ is compact \cite[Def. 1.3.]{Tem97} and the support of $\mu$ is contained in $\mathcal{A}$; see {\it e.g.} \cite[Lemma 5.1]{chekroun_glatt-holtz} for this latter point. Combined with the assumption on $f$, we deduce 
that $\varphi \in L^1_{\mu}(\mathcal{A})$ and {that it constitutes therewith} an admissible observable.

Applying \eqref{Eq_phys_rev} along with the definition of the Koopman operator in \eqref{Def_Koop} 
{yields}
\be
\underset{T\rightarrow \infty}\lim \frac{1}{T} \int_0^{T}  \| \mathfrak{L} U_t v( y)-f(U_t v ( y))\|^2 \d t =\int_{\mathcal{A}} \| \mathfrak{L} v ( y)- f(v( y)) \|^2 \d \mu ( y),
\ee
for almost all $ y \in Y$. From  \eqref{Eq_Liouv} and \eqref{Def_expec}, we obtain next that, for almost all $ y \in Y$,
\beas
\underset{f \in \mathcal{E}(v,\mu)}{\inf} \Big(\underset{T\rightarrow \infty}\lim \frac{1}{T} \int_0^{T}  \|\frac{ \mathrm{d} x}{ \mathrm{d} t}-f(x(t; y))\|^2 \d t \Big) & = \underset{f \in \mathcal{E}(v,\mu)}{\inf} \Big(\int_{\mathcal{A}} \| \mathfrak{L} v ( y)- f(v( y)) \|^2 \d \mu ( y)\Big),\\
& = \mathbb{E}\Big[\big\|\mathfrak{L} v-\mathbb{E}\big[\mathfrak{L} v\big| v\big]\big\|^2\Big],
\eeas
where \eqref{Eq_best_app} has been used to get the last equality.

Finally, by applying in the Hilbert space $L^2_{\mu}(Y;X)$ the classical {projection theorem} 
onto a closed convex set  \cite[Theorem 5.2]{brezis_book}, we  conclude that $\mathbb{E}\big[\mathfrak{L} v\big| v\big]$ is the unique minimizer of 
\be
\int_{\mathcal{A}} \| \mathfrak{L} v ( y)- f(v( y)) \|^2,
\ee
 and, therefore, of 
 \be 
 \underset{T\rightarrow \infty}\lim (1/T) \int_0^{T}  \| \mathfrak{L} U_t v( y)-f(U_t v ( y))\|^2 \d t,
 \ee
  when $f\in \mathcal{E}(v,\mu).$ Formula \eqref{Eq_key} is thus proved.
\end{proof}

\br\label{rem:relax} 
{One may want to assume the existence of a bounded non-wandering set $\Lambda$\cite{Chek_al13_RP}, instead of a global attractor ${\mathcal A}$}, for the system \eqref{Eq_main}  and thus relax condition \eqref{Eq_phys_rev} 
to hold only for $y$ in the system's basin of attraction $\mathcal{B}(\Lambda)$. {In this more general case,} the conclusion of Lemma \ref{Main_lemma} still holds for $y_0 \in \mathcal{B}(\Lambda)$, 
{i.e.,} for time series that relax towards the corresponding statistical equilibrium $\mu$ whose support is contained in $\Lambda$. MSMs can still be efficiently derived in such a context as 
{shown in Sec.~\ref{sec_pop_dyn}.} 
\er 

\section{MSMs as closure models from time series: the $\eta$-test}
\label{Practical_MZ}

A classical approach in MZ modeling consists of assuming something reasonable about the statistics of  the unobserved variables --- {\it e.g.}, on the basis of previous observations --- and then turn tothe formulation of Eq.~\eqref{Eq_MZ_GLE} to determine an analytic or numerical approximation of the terms  
that appear on the right-hand side; the resulting prediction methods based on the 
MZ formalism go under the name of {\it optimal prediction}  \cite{Chorin_MZ}. The main objective in such an approach is to calculate  the conditional expectation based on the assumptions made regarding the statistics of  the unobserved variables. When no analytical formula is available for the probability measure that describes the distribution of the unobserved variables, empirical estimation methods are typically used; these empirical estimations  
often rely on a large set of trajectories integrated over a short time interval.

For instance, maximum likelihood estimation techniques \cite{lehmann1998theory} can be used to find an approximation $\nu$ of the density of the unobserved variables as Gaussian mixtures; see \cite{Stinis06} for an application of such techniques to the non-Hamiltonian case of the Kuramoto-Sivashinsky equation. In certain cases, the Markovian term in \eqref{Eq_MZ_GLE} can be computed explicitly by relying on the estimated $\nu$, while other assumptions --- such as the {\it short-memory approximation} or the  {\it $t$-model} \cite{hald_stinis,stinis_euler} --- can then be used to deal with the non-Markovian term in \eqref{Eq_MZ_GLE} through simulations of the full system; see \cite{Chorin_Hald-book, Stinis06}. 

{As mentioned already in Sec.~\ref{ssec:outline}, the} difference in viewpoint between this classical approach and an approach based on averaging along trajectories, such as supported by Lemma  \ref{Main_lemma}, is similar to the Eulerian versus the Lagrangian viewpoint in fluid mechanics. {In this analogy,} the 
approach advocated in this article corresponds to the Lagrangian viewpoint: 
measurements are made along trajectories and 
the numerical construction of an MZ model 
{relies on} fewer initial states, but requires longer runtimes, for large models, or longer data records, for observational data. Recalling that the MZ formalism is built on the  decomposition of the Koopman semigroup given in \eqref{Eq_MZ_decomp}, an approach based on averaging along trajectories is  furthermore consistent with other Koopman operator techniques developed recently for the spectral analysis of time series \cite{budivsic2012applied}.

The MSM approach to stochastic inverse modeling, as described above,  
 can thus provide an efficient way of deriving approximation of \eqref{Eq_MZ_GLE} in practice, by relying exclusively on available time series, even when no prior knowledge about the full model is available. A key point to 
achieving this, 
based on available time series alone, 
is the quality of the approximation by the vector field $f_1$ in \eqref{Eq_MSMfunc2} of the genuine Markovian contribution in \eqref{Eq_MZ_GLE}, on the one hand, and the quality of approximation by the terms collectively labeled (b) in Eq.~\eqref{Eq_MSMfunc2} of the non-Markovian contribution $\int_{0}^t U_{t-s} \mathcal{G}(v;\eta_s)\d s$ in \eqref{Eq_MZ_GLE}, on the other. Recall that the former  
model the self-interactions among the observed variables and that the latter 
model the cross-interactions between the observed and unobserved variables that occurred at 
{ the past times $s$, 
$0 \le s \le t$.}

{\it A priori} error estimates that 
{may} be useful in practice are difficult to establish at this level of generality. However, the remaining term (c) in Eq.~\eqref{Eq_MSMfunc2} --- when compared to the ``residue" $\eta_t$ in \eqref{Eq_MZ_GLE}  --- can serve to formulate a correlation-based criterion to help 
{determine how well an MSM given by \eqref{Eq_MSMfunc2} approximates} 
\eqref{Eq_MZ_GLE}. 
Indeed, by construction of the GLE, the noise term $\eta_t$ is 
uncorrelated with the observed time series $x(t)$, 
{due to the orthogonality property of the dynamics in} \eqref{Eq_MZ_ortho2}; see also Remark \ref{Rem_ortho}. Therefore, the corresponding term (c) in \eqref{Eq_MSMfunc2} can naturally serve for testing whether an MSM derived from the time series $x(t)$ alone provides a good approximation of the GLE.

{The resulting} {\it $\eta$-test} can be summarized as follows:  
\begin{itemize}
\item[($\mathfrak {T}$)]  
The more  Pearson's correlation coefficient\footnote{{ This coefficient is} defined as the covariance of the two variables divided by the product of their standard deviations.} between $\xi(t,\omega)$ in \eqref{Eq_MSMfunc2} and the observed variable $x(t)$ is close to zero, the 
better the approximation of the 
 GLE~\eqref{Eq_MZ_GLE} associated with $x(t)$. 
\end{itemize}
In particular, to assert that an MSM model constitutes a good approximation of the   
\eqref{Eq_MZ_GLE} associated with a given multivariate time series,  the noise terms labeled (c) in Eq.~\eqref{Eq_MSMfunc2} should not exhibit any $x$-dependence, since the $x$-dependence of the fluctuating terms is supposed to be taken into account { solely} in the non-Markovian terms (b). 

Interestingly, the multilayer structure of an MSM provides a simple way to compute the (c)-term $\xi(t,\omega)$ of Eq.~\eqref{Eq_MSMfunc2}, as provided by the representation formula  \eqref{Eq_xi_interp} in Proposition \ref{Main_Prop}. In that respect, the process $z_1$ can be easily simulated by integration of \eqref{Eq_z1}, 
followed by an integration of 
\be
\dot{\xi}=f_2(\xi)+\Pi z_1.
\ee
{Doing so} provides a natural estimation of $\xi(t,\omega)$, up to multiplication by the matrix $C$ from \eqref{Eq_g1}, and thus allows for an easy estimation of 
Pearson's correlation coefficient 
{in} the $\eta$-test ($\mathfrak {T}$) formulated above.

It is important to keep in mind that the $\eta$-test 
{only provides information} on the $x$-dependences 
{of this residue} that are not fully 
{captured} by the (a)- and (b)-terms in \eqref{Eq_MSMfunc2}. 
{Hence} this test is more useful as an indicator in the design of the relevant constitutive parts of \eqref{Eq_MSMgen} such as the { $f_i$'s and the $g_i$'s, $i = 1,2$,} rather than providing an 
{ultimate criterion} to assess the modeling performance of a given  MSM.  

Situations may indeed arise where the corresponding Pearson's correlation coefficient is not necessarily close to zero, while the corresponding MSM still performs very well in 
simulating 
the main statistical properties of the observed variables. In other words, although some $x$-dependences may not be fully resolved by the { deterministic 
terms, both Markovian and non-Markovian, 
of} an MSM, the contributions of these terms to 
simulating the main observed statistics could prove to be negligible. 

Such a situation   is identified in Sec.~\ref{sec:Climate_ex} below for a conceptual stochastic climate model in the presence of weak time scale separation; 
see panels (g) and (h) of Figs.~\ref{fig:pdf1d} and \ref{fig:pdf2d}, and panel (d) of Fig.~\ref{fig:emr_acf}  in the following section.  At the same time, the quality of reproduction of the observed statistics is improved as 
Pearson's correlation coefficient gets closer to zero; see panels { (a)--(f) of} Figs.~\ref{fig:pdf1d} and \ref{fig:pdf2d}, and panels { (a)--(c) of} Fig.~\ref{fig:emr_acf}. Part of the reason for this success lies in the energy-conserving { EMR} formulation discussed at the end of  Sec~\ref{Sec_MSMs_att} 
and in Appendix~\ref{App_energy}.  
Once adopted, this formulation makes the 
practical, discrete-time EMR consistent with the full model's structural features, as layed out in Sec.~\ref{sec:Climate_ex}.


\section{A conceptual stochastic climate model: Numerical results}
\label{sec:Climate_ex}

\subsection{ Model formulation}
\label{ssec:clim_model}
 We illustrate here the MSM approach to stochastic inverse modeling, as described 
in Secs.~\ref{sec:emr_form}--\ref{Practical_MZ}, by deriving  
a closure model from partial observations of a slow-fast system in which only the nominally slow variables are observed. 
The time-scale separation between the nominally slow and fast variables ranges from a strong 
to a weak separation; 
 in the latter case, some of the slow and fast variables actually evolve on a similar time scale. This example will show, in particular, the usefulness of the $\eta$-test introduced in Sec.~\ref{Practical_MZ}, as discussed at the end of this section.  
For simplicity and for the sake of reproducibility of the results, {we use here} a simple conceptual  climate model 
proposed in \cite{Majda_etal2005}; see also  \cite{Franzke_etal07, Majda_etal2008}. Three 
features of this four-dimensional model are  
 of interest with respect to our closure problem ($\mathfrak{P}$) 
of Sec.~\ref{sec:emr_form}.  
First, the model is stochastic, which introduces {\it de facto} noisy observations. Second, the variables that will be taken as unobserved, carry  in fact most of the variance in this model; { while} this is not 
the case in { actual} observations of atmospheric low-frequency variability (LFV) \cite{Kim_Ghil'93a, Kim_Ghil'93b}, { it does present} a challenging difficulty to the MSM approach.  Third, the model exhibits natural energy-preserving constraints that will be taken into account in the MSM formulation below.

The model obeys the following system of SDEs:
 \begin{subequations}\label{toy_model}
 \begin{align}
 & {\rm d}{x_1} = \{   -x_2 (L_{12} +a_1x_1+a_2x_2) - d_1x_1 +F_1 + {\bf L_{13}y_1+ b_{123}x_2y_1 + c_{134}y_1y_2} \}{\rm d}t, \label{toy_1} \\
 & {\rm d}{x_2} = \{     x_1 (L_{21} +a_1x_1+a_2x_2) - d_2x_2 +F_2 + {\bf L_{24}y_2+    b_{213}x_1y_1} \}{\rm d}t,\label{toy_2} \\
 & {\rm d}{y_1} = \{ -L_{13}x_1+ b_{312}x_1x_2 +c_{341}y_2x_1+ F_3-\frac{\gamma_1}{\epsilon}y_1\}{\rm d}t + \frac{\sigma_1} {\sqrt{\epsilon}}dW_1 \label{toy_3} \\
 & {\rm d}{y_2} = \{ -L_{24}x_2 + c_{413}y_1x_2+ F_4-\frac{\gamma_2}{\epsilon}y_2\}{\rm d}t + \frac{\sigma_2}{\sqrt{\epsilon}}dW_2. \label{toy_4}
 \end{align}
 \end{subequations}
The parameter $\epsilon$  explicitly controls the time-scale separation between the model's slow and fast variables, namely the $(x_1, x_2)$-variables, and the $(y_1, y_2)$-variables, respectively. These variables  are 
coupled  linearly, through the skew-symmetric terms, as well as nonlinearly;  the nonlinear coupling involves the triple coefficients $b_{ijk}$ and $c_{ijk}$.  The linear and nonlinear coupling can be understood as additive and multiplicative noise forcing the slow-mode evolution, respectively.  

The model set-up here follows \cite{Franzke_etal07}, namely:
\beas
& b_{123}=b_{213}=0.25, \, b_{312}=-0.5,\\
& c_{134}=c_{341}=0.25, \, c_{413}=-0.5;
\eeas
\beas
& L_{12}=L_{21}=1, \, L_{24}=-L_{13}=1,\\ 
& a_1=-a_2=1.0, \, d_1=0.2, \, d_2=0.1; 
\eeas
and
\beas
& F_1=-0.25, \, F_2=F_3=F_4=0;\\ 
& \gamma_1=\gamma_2=1, \mbox{ and }  \sigma_1=\sigma_2=1. 
\eeas
Note that, in agreement with the energy-conserving constraints of the EMR formulation in Eqs.~\eqref{nlcons1}--\eqref{poscons} 
of \ref{App_energy}, this toy model has a quadratic nonlinear part that conserves energy;  for example, the triple coefficients sum to zero, i.e.,
\bes
b_{123}+b_{213}+b_{312}=0 \mbox{ and }  c_{134}+c_{341}+c_{413}=0,
\ees
as in Eq.~\eqref{nlcons3}, while the linear part has pairwise skew-symmetric terms given by the values of the coefficients $L_{12}, L_{21}, L_{24}$ and $L_{13}$, as in Eq.~\eqref{skewcons}. Furthermore, the negative-definite contributions of $\gamma_1$, $\gamma_2$, $d_1$ and $d_2$, as in Eq.~\eqref{poscons}, ensure the model's dissipativity, cf.~\cite{schenk1998random} and \cite[Sec.~5.4]{GhCh87}. 

Finally,  certain quadratic terms are absent from the model, as per Eq.~\eqref{nlcons1}, while the values for $a_1$ and $a_2$ are set so that Eq.~\eqref{nlcons2} is satisfied. The terms in bold characters in the first two equations will be discussed further below.

\subsection{Numerical results}
\label{ssec:numerical}

We integrated Eqs.~\eqref{toy_1}--\eqref{toy_4} for $10^4$ time units by using the fourth-order Runge--Kutta scheme for the deterministic part and the Euler-Maruyama scheme for the stochastic part, with a time step of $\Delta t=0.001$. 
Only the slow model variables $x_1$ and $x_2$ are stored here, with a sampling rate of $0.05$ time units. We applied 
an energy-preserving version of \eqref{Eq_EMR}, by using the constraints of Eqs.~\eqref{nlcons1}--\eqref{poscons}, in order to model the corresponding multivariate  time series of $(x_1,x_2)$. 

As the scale-separation parameter $\epsilon$ increases, and the scale separation 
decreases {therewith} in the model, the decorrelation times steadily increase for the fast modes and decrease for the slow modes; {$x_1$ and $y_1$, in particular, exhibit the} most pronounced changes; { see Figs.~\ref{fig:data_acf}(a)--(d).} For $\epsilon=1.5$, 
the autocorrelation functions for $y_1$ and $x_1$ become very similar, {cf. Fig.~\ref{fig:data_acf}(d),} and so there is no longer any formal separation of scales.   
\begin{figure}
\vspace{-2ex}\centering
\includegraphics[height=0.65\textwidth, width=.85\textwidth]{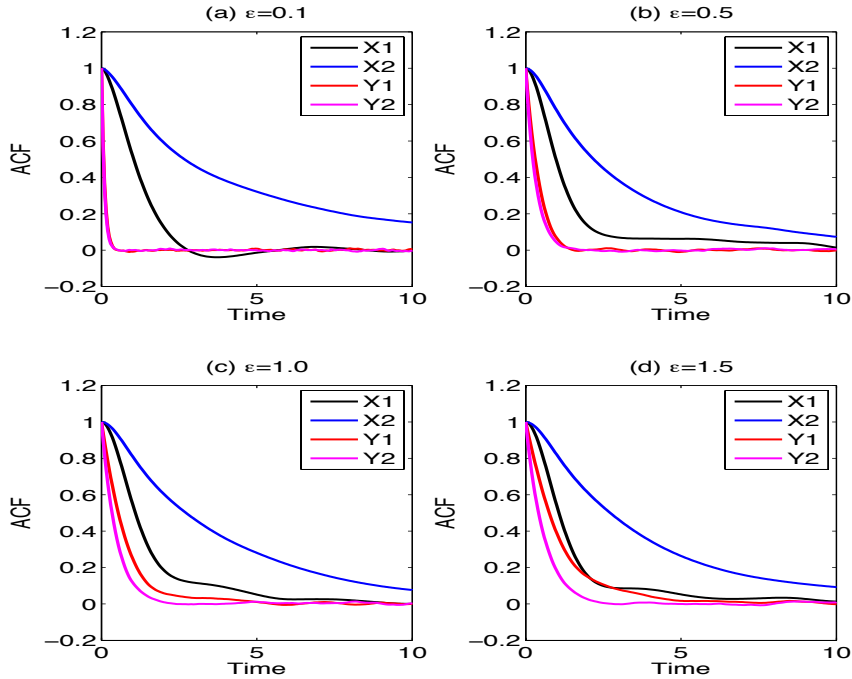}
\vspace{-5ex}
\caption{\small
Autocorrelation functions of the model variables for Eq.~\eqref{toy_model}. Panels (a)--(d) correspond to the $\epsilon$-values $0.1, 0.5, 1.0$ and $1.5$; { see the} color coding for the variables $(x_1, x_2, y_1, y_2)$ in the legend of each panel.
\label{fig:data_acf}}
\end{figure}
Note that, in this case, the main level of our resulting EMR model for evolving $x_1$ and $x_2$ does not include explicitly  the unresolved, linear and nonlinear interactions marked in bold in  Eqs.~(\ref{toy_1}, \ref{toy_2}). These  contributions to the dynamics of the observed variables  
{$(x_1, x_2)$, which result from their} cross-interactions with the 
unobserved variables 
{$(y_1, y_2)$,} need to be properly parameterized by the EMR { model's}  
hidden variables $\{r_t^{(m)}: 0 \le m \le p\}$, in order to reproduce 
the statistical behavior of 
{$(x_1, x_2)$} in terms of their 
{PDFs} and autocorrelations.

The energy-preserving EMR model, fitted solely on the $x_1$ and $x_2$ time series, has two additional levels ($p=2$) for all values of $\epsilon$, according to the stopping criterion of \ref{App1}. Figures \ref{fig:pdf1d} and \ref{fig:pdf2d} present a comparison of {the one- and two-dimensional (1-D and 2-D) PDFs,} respectively, for slow modes obtained by the energy-preserving  EMR and the full model. The two figures show clearly that the energy-preserving   EMR model reproduces
quite accurately both the univariate and bivariate PDFs.

\begin{figure}
\vspace{-2ex}\centering
\includegraphics[height=0.8\textwidth, width=.9\textwidth]{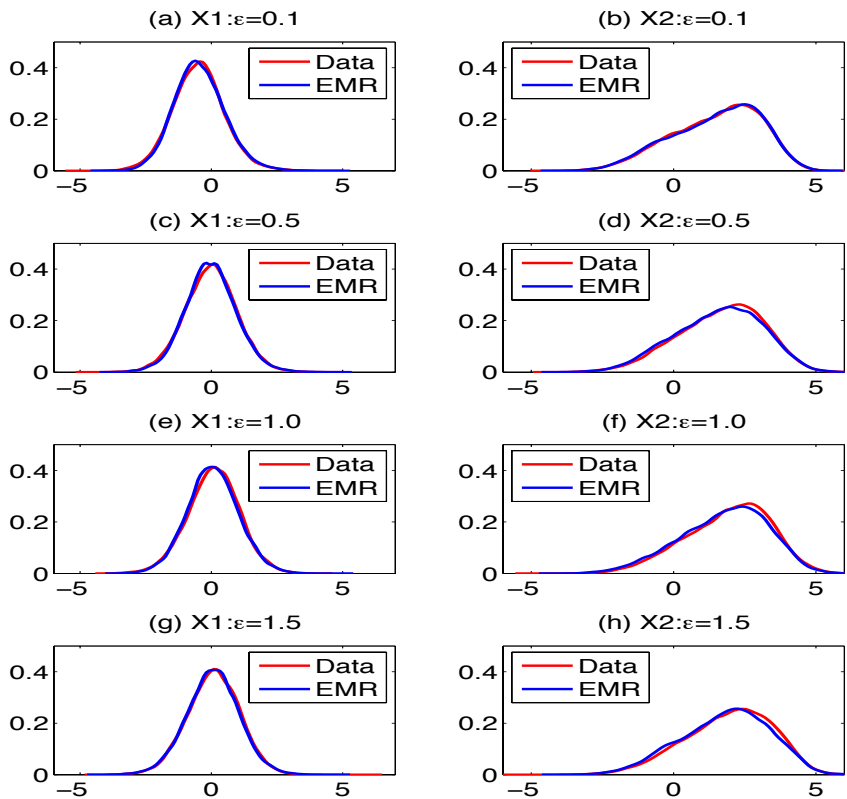}
\vspace{-7ex}
\caption{\small
One-dimensional (1-D) probability density functions (PDFs) of the resolved variables $(x_1, x_2)$, as modeled by the EMR-reduced  model (blue line) {\it vs.} { the simulation} by the full coupled system in Eq.~\eqref{toy_model} (red line); the panels differ, from top to bottom, by the model's varying scale separation: {$\epsilon = 0.1$ in panels (a) and (b); 0.5 in panels (c) and (d); 1.0  in panels (e) and (f);  and 1.5  in panels (g) and (h).} Left column (a, c, e, g) PDFs for $x_1$; and right column (b, d, f, h) PDFs for $x_2$.
\label{fig:pdf1d}}
\end{figure}
\begin{figure}
\vspace{-2ex}\centering
\includegraphics[height=0.7\textwidth, width=1\textwidth]{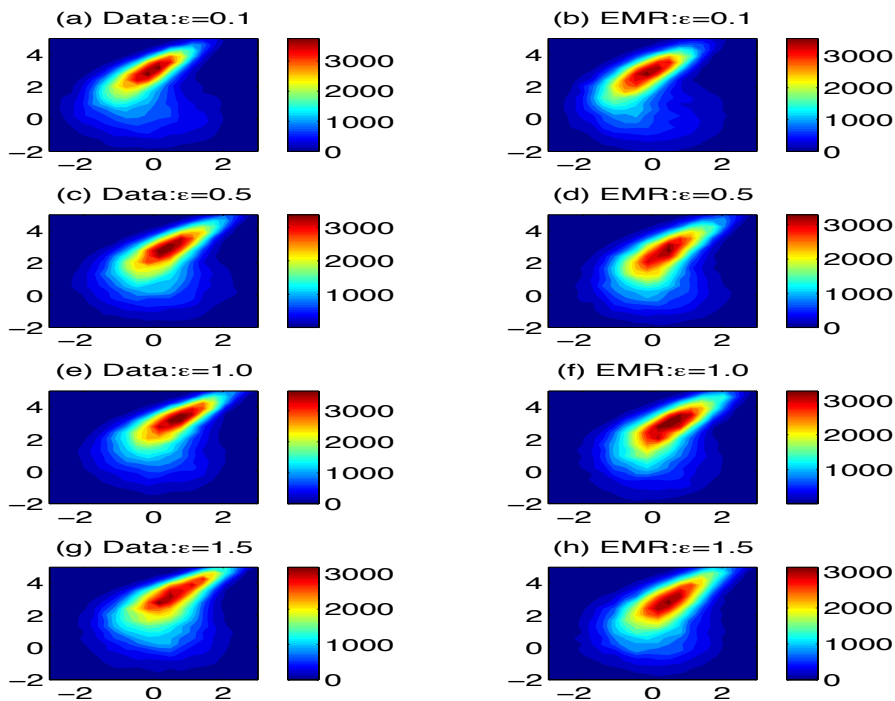}
\vspace{-8ex}
\caption{\small
Same as Fig.~\ref{fig:pdf1d}, but for the two-dimensional (2-D) PDFs of the resolved variables $(x_1, x_2)$, as simulated by the full coupled system  of Eq.~\eqref{toy_model} (left column) {\it vs.} the { 2-D} PDFs of $(x_1,x_2)$, as modeled by the EMR-reduced model (right column).
}\label{fig:pdf2d}
\end{figure}
\begin{figure}
\includegraphics[height=0.5\textwidth, width=.6\textwidth]{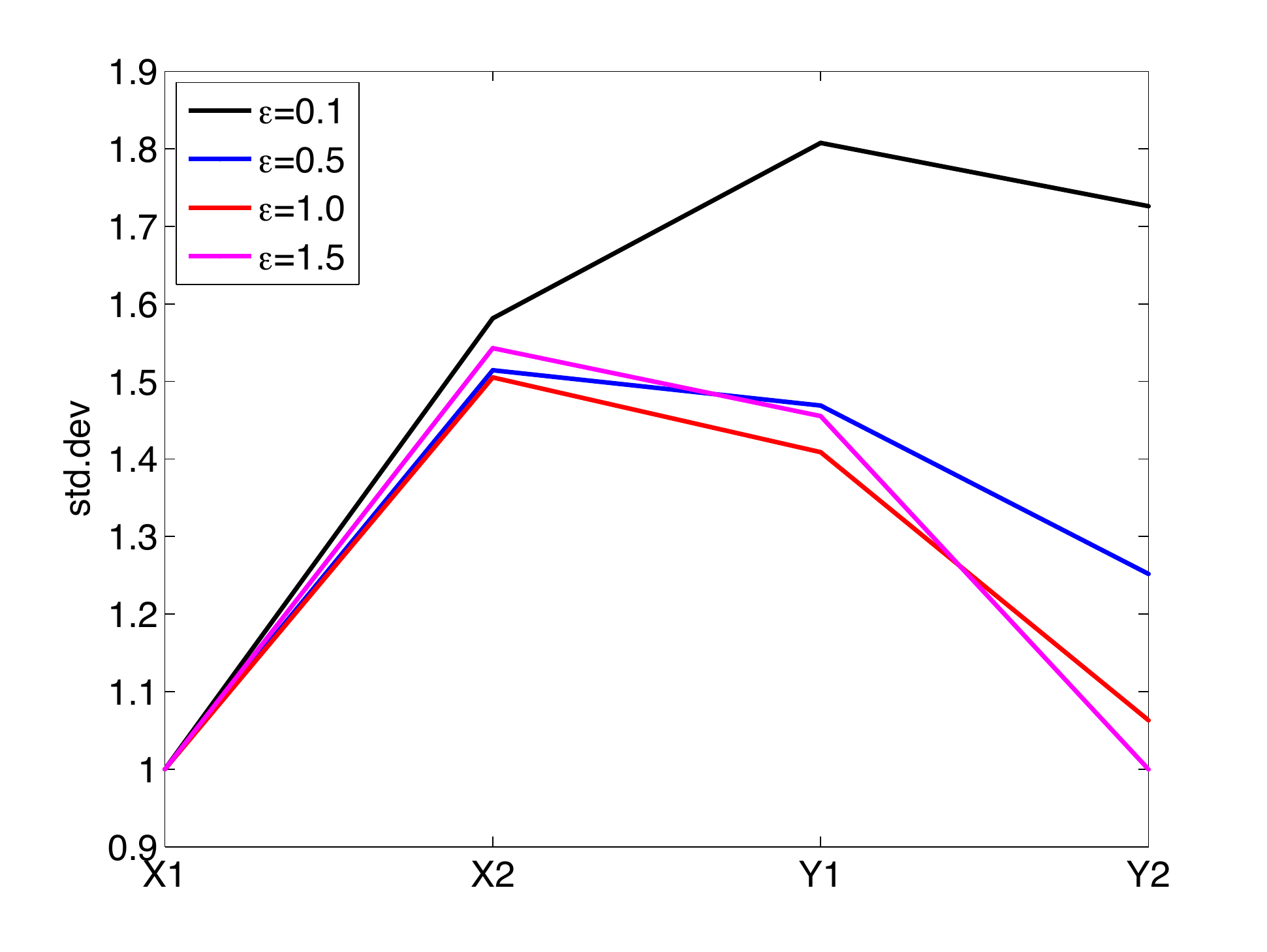}
\vspace{-2ex}
\caption{\small
Standard deviation of the model variables of the toy model governed by Eq.~\eqref{toy_model}, for different values of the scale-separation parameter, namely $\epsilon = 0.1, 0.5, 1.0$ and $1.5$. The slow modes $(x_1, x_2)$ and fast modes $(y_1, y_2)$ are on the abscissa; see legend for color code. 
}\label{fig:data_var}
\end{figure}
\begin{figure}[htbp]
\vspace{-2ex}\centering
\includegraphics[height=0.6\textwidth, width=.85\textwidth]{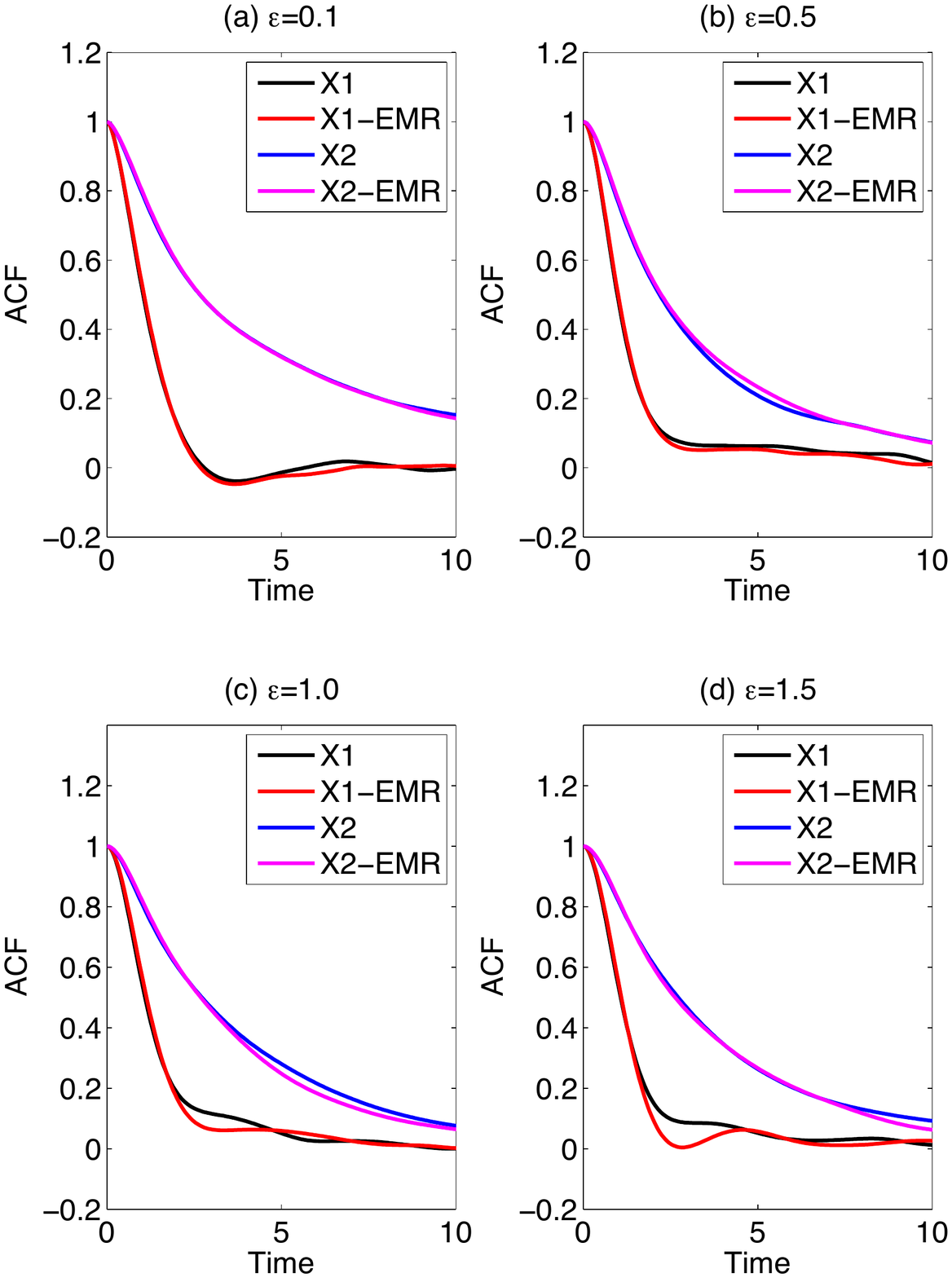}
\vspace{-6ex}
\caption{\small Autocorrelation functions of the resolved variables $(x_1, x_2)$ in the full dynamics {\it vs.} the EMR-reduced dynamics, with varying scale separation $\epsilon = 0.1 - 1.5$ in panels (a)--(d). Color coding appears in each panel legend: the black and blue lines are for the full  model's first and second component, while 
red and purple are for the EMR-reduced model's first and second component.
}
\label{fig:emr_acf}
\end{figure}

Figures~\ref{fig:data_acf} and \ref{fig:data_var}  show how {the autocorrelations and the variance,} respectively, of the slow and fast variables change when the scale-separation parameter $\epsilon$ varies between $0.1$ and $1.5$. For  $\epsilon=0.1$ {(black line in Fig.~\ref{fig:data_var}),} most of the variance is carried by the fast modes, {whose decorrelation time is much smaller in this case than that of the} slow modes, while $x_2$ has the slowest {autocorrelation decay, cf. Fig.~\ref{fig:data_acf}(a).} For $\epsilon \ge 0.5$ {(colored lines in Fig.~\ref{fig:data_var}),} the variance {in this model} is distributed roughly equally between slow and fast modes, while observations of atmospheric LFV show the slow, barotropic modes, on the 10--100-day scale, to be in fact considerably more energetic than the synoptic-scale, baroclinic modes of 1--10 days \cite{Kim_Ghil'93a, Kim_Ghil'93b}.

The 1-D PDF is mostly Gaussian for $x_1$ { (left column of Fig.~\ref{fig:pdf1d}) and strongly non-Gaussian for $x_2$ (right column of Fig.~\ref{fig:pdf1d}); 
neither changes} noticeably with scale separation.  The 2-D PDFs in Fig.~\ref{fig:pdf2d} are quite non-Gaussian, {all four of them,} and they do not change much in shape or orientation with $\epsilon$ either.  
Moreover, {Fig.~\ref{fig:emr_acf} shows that the energy-preserving  EMR model} reproduces with very high accuracy changes in {the autocorrelation function of $x_2$, the model's} slowest mode, 
for all values of $\epsilon$. {The energy-preserving  EMR model does equally well for $x_1$, as long as the scale separation is sufficiently large, i.e. for $\epsilon=0.1$ and $0.5$; see { Figs.~\ref{fig:emr_acf}(a, b).} For $\epsilon=1.0$ and $1.5$, there is no pronounced scale separation 
in the full model, and thus there is not much difference in decorrelation time between the slow mode $x_1$ and the fast mode $y_1$. Hence it is not surprising that the autocorrelation for $x_1$ is reproduced somewhat less accurately, but still reasonably well, {cf. Figs.~\ref{fig:emr_acf}(c,d).}

In summary, in {the partially observed situation studied in this section,} a discrete-version of an MSM given by \eqref{Eq_EMR} and subject to the energy-preserving constraints described in \ref{App_energy} performs very well when the variability of the discarded variables is much faster than that of the slow variables. The EMR model performance is still remarkably good when the variability of the excluded variables is similar in amplitude and time scale to that of the retained variables. 

{Since the number of levels in the EMR model is $p=2$, the total number of EMR variables is six, and thus it} formally exceeds the total number of degrees of freedom of the full model Eq. \eqref{toy_model}, namely four.
This is the price to pay, though, for successful orthogonal, multilevel parameterization of the unresolved processes that were explicitly excluded from the main level of the reduced model. 

As discussed in Section~\ref{Practical_MZ}, it can reasonably be asserted that an MSM written under its form \eqref{Eq_MSMfunc2} represents a good approximation of the 
{GLE --- i.e.,} the optimal closure model predicted by the  MZ formalism ---
{as long as} the noise  term 
labeled (c) in { Eq.~\eqref{Eq_MSMfunc2}} is weakly correlated in time with $x$;  see the $\eta$-test $({\mathfrak T})$ there. 

For the  model \eqref{toy_model} at hand, the maximum absolute values of the corresponding component-wise Pearson's correlation coefficients  are $0.11, 0.33, 0.42$ and $0.47$, for $\epsilon=0.1, 0.5, 1.0$ and $1.5$, respectively. As a consequence, the $\eta$-test allows { one} to conclude that the more the time-scale separation  is reduced, the more some of the $(x_1,x_2)$-dependences (in terms of correlations)  are not well  resolved by the Markovian and non-Markovian terms conveyed by our energy-preserving  \eqref{Eq_EMR}. 
However, as illustrated here, the contributions of such dependencies  to a good reproduction of the main observed statistics may turn out in practice to be negligible.
{When that is the case, the resulting EMR model may still 
perform quite well;} see panels (g) and (h) of Figs.~\ref{fig:pdf1d} and \ref{fig:pdf2d}, and panel (d)  of Fig.~\ref{fig:emr_acf}. 

Finally, as predicted by the $\eta$-test  proposed in Sec.~\ref{Practical_MZ} here, the reproduction of the observed statistics is improved as Pearson's correlation coefficient gets closer to zero; see panels (a)--(f) of Figs.~\ref{fig:pdf1d} and \ref{fig:pdf2d}, and panels (a)--(c) of Fig.~\ref{fig:emr_acf}.  Part of the reason for this success, in the case at hand, lies in the energy-conserving formulation discussed at the end of Sec.~\ref{Sec_MSMs_att} and in \ref{App_energy}.

Recall, though, as stated already in Sec.~\ref{ssec:background}, that conservation of quadratic invariants, like in Eq.~\eqref{toy_model} here, is the exception --- rather than the rule --- in full climate models that contain thermodynamic variables, precipitation and chemistry, let alone in other areas of the physical and life sciences. In the next section, we consider a population-dynamics model that does contain quadratic coupling terms between variables but does not possess a quadratic energy to be conserved, {in absence of dissipation}.


\section{Reflected diffusion processes by MSMs and reconstruction of strange attractors}\label{sec_pop_dyn}

\subsection{The original model}
Our second  data-based closure example is the following 
classical population dynamics model:
\be\label{LV_model}
\frac{\d N_i}{\d t}=b_i N_i\Big(1-\sum_{j=1}^n a_{ij} N_j\Big),   \quad 1\leq i\leq n;
\ee
{here $N_i$} denotes the population size of the $i^{th}$ species 
{relative} to its carrying capacity, $b_i > 0$ denotes its intrinsic growth rate, and { the  $a_{ij}$'s} denote the 
interaction coefficients: {intraspecific when $i=j$, and interspecific when $i\neq j$.} We restrict ourselves to the case where $a_{ij}\geq 0$, {which corresponds} to the  well-known {\it competitive Lotka-Volterra system} \cite{volterra1931leccons, may2001stability}. Such systems manifest 
certain generic features that make their study {even more 
interesting.} For instance, they generate flows that  \textemdash\, when restricted to the appropriate invariant sets  \textemdash\, are topologically equivalent to a broad class of flows generated by systems of first-order ODEs with polynomial right-hand sides, such as the Lorenz system \cite{peschel1986predator,kozlov2013chaos}. 

This class of simple models exhibits a rich variety of dynamics, depending on the parameter 
values, as long as the initial vector  $\mathbf{N}_0$ lies in $\mathbb{R}_+^n:=\{N_i : N_i \geq 0, \,  i = 1 \ldots, n\}$, although solution blow-up is possible when a component of $\mathbf{N}_0$ is negative.
 Smale showed in 
\cite{smale1976differential} that --- under the aforementioned conditions and for any initial data in the cone $\mathbb{R}_+^n$ --- a Lotka-Volterra system \eqref{LV_model} of five or more species ($n \geq 5$) can exhibit any asymptotic behavior, including steady states, limit cycles, $n$-tori, or more complicated attractors. This result has had a profound influence on the theory of monotone dynamical systems in showing that competitive systems could display more than simple dynamics; see \cite{hirsch1982systems, hirsch1985systems,hirsch1988systems,hirsch1990systems, smith2008monotone}.

In particular, the proof of Smale ensures  the existence of  a closed invariant set $\mathcal{C}$ which is homeomorphic to the $(n-1)$-dimensional simplex 
\be\label{Eq_simplex}
\Delta_{n-1}=\{N_i: N_i\geq 0, \; \sum_{i=1}^n N_i=1\},
\ee
which is attracting every point (excluding the origin) in the domain $\mathbb{R}_+^n$; see \cite[pp.~71-72]{smith2008monotone}. The set $\mathcal{C}$ is known as the
carrying simplex since it ``carries'' all of the asymptotic dynamics associated with \eqref{LV_model}, and its
existence significantly reduces the set of possibilities in certain dimensions. For instance, the existence of such a simplex implies that an attractor associated with \eqref{LV_model} cannot have a dimension greater than $n-1$ so that, in particular, chaos cannot take place when $n=3$. This { situation} is obviously 
in stark contrast with 
that of three-dimensional ODE systems with quadratic energy-preserving nonlinearities, such as the standard Lorenz system {\cite{Lorenz63, GhCh87}, although} the dynamics of the latter 
{is still realizable  
in a higher-dimensional} Lokta-Volterra system; 
see \cite[Sect.~4]{kozlov2013chaos}.  

In fact, as numerically shown in \cite{arneodo1982strange}, the smallest dimension 
{for which} complicated dynamics takes place on a strange attractor for \eqref{LV_model} corresponds to $n=4$.   As already conjectured in \cite{arneodo1982strange}, homoclinic tangencies are responsible 
{for} the structural instability of such strange attractors \cite{vano2006chaos}; 
{the latter} can even experience sudden changes, and be  brutally transformed into a steady state, after 
{a small change in parameter 
values;} see \cite[Figs.~3 and 4]{RC11}. Another interesting feature of \eqref{LV_model} for $n=4$ is the rarity of occurrence of chaotic { behavior in a} 20-dimensional parameter space; see 
\cite{vano2006chaos} and \cite[Fig.~2]{RC11}.

Furthermore, when chaos takes place for a particular set of parameters, the corresponding strange attractor may not attract all the points of  $\mathbb{R}_+^4$; 
 in the latter case, it typically coexists 
with simple local attractors, such as fixed points, 
and the attractor-basin boundaries are fractal, 
{as in}~\cite[Fig.~5]{RC11}.}
Finally, for a solution evolving on a typical strange attractor,  
the auto-correlation function  of each component decays at a nearly identical rate; see, 
for instance, the red curves in Fig.~\ref{Fig_ACF_LV} below.

The aforementioned features  \textemdash\,  lack of time-scale separation; rarity of strange attractors in the parameter space; existence of fractal 
{boundaries between attractor basins;} and positivity of { the solutions'} components \textemdash\,  
{greatly add to the difficulties in 
deriving} a data-based stochastic closure model able to simulate faithfully the statistics of  solutions that 
evolve on a strange attractor associated with \eqref{LV_model} for $n=4$; { and especially so when only using a time series that represents partial observations of such a solution.}
The stumbling block of lack of time-scale separation 
has {already been} discussed. 
The combination of the two facts that (i) in general a time series of partial observations of solutions evolving on a strange attractor  can be rigorously represented as a stochastic process 
{that does depend} on the unobserved variables (see Theorem A in \cite{Chek_al13_RP} and Corollary B in { its} supporting information; and that (ii) chaos takes place over very small regions of the parameter space for \eqref{LV_model} with  $n=4$, seriously hamper the learning of any inverse stochastic model 
{for} the statistics of such partial observations.  

Difficulties  in estimating the correct statistical behavior can already be observed in the case of full observations of a chaotic attractor \cite{voit2010parameter}: 
when the corresponding  four-dimensional time series is corrupted by a Gaussian noise, the accuracy of the parameter values  \textemdash\, as estimated by simple regression \textemdash\,  starts 
to degrade as the data become noisier \cite[Table 1]{voit2010parameter} and the 
estimated behavior may quickly deviate from the original one; see   \cite[Fig.~3]{voit2010parameter}. This divergence is not surprising 
{and has to be} overcome with any other estimation approach, 
{given} the rarity of occurrence of chaos in the { model's} parameter space. 

We show in the next subsection that --- for the more challenging case of partial observations 
\textemdash\,  
{it suffices to add} a positivity constraint to the standard EMR formulation \eqref{Eq_EMR} 
{for our MSM approach to allow one 
to derive} closure models with very good  statistical simulation skill.


\subsection{Numerical results}

The parameters of the system \eqref{LV_model} for $n=4$ read as follows:

\be\label{A}
(a_{ij})_{1\leq i,j\leq 4}=\left( \begin{array}{cccc}
 1 & 1.09  &  1.52 & 0 \\
 0 & 1 &  0.44 &  1.36 \\
  2.33 &  0 &  1 & 0.47 \\
   1.21 &  0.51 & 0.35 & 1
\end{array} \right), \;\;  b=\left(\begin{array}{c} 1\\ 0.72 \\ 1.53\\  1.27 \end{array} \right).
\ee

These parameter values are those used in \cite{vano2006chaos, RC11}, for which numerical evidence of chaos is well established.
Time series of length 
{$l=1.5 \times 10^5$} were generated by integration of \eqref{LV_model}, 
{using} a standard Euler scheme with time step $\delta t= 0.035$ and 
initialized at
\be
\mathbf{N_0}=(0.5, 0.2, 0.3, 0.7)^T.
\ee
For such an initial 
{state}, the time series evolves on an approximation $\widehat{\mathcal{A}}$ of the strange attractor $\mathcal{A}$; see left panel of Fig.~\ref{Fig_Att_LV} below.
Only the 
the first three components $(N_1,N_2, N_3)$ are observed from the integration of this four-dimensional system of ODEs, after removal of the  initial transient.

In this {population dynamics} context, the first property any inverse stochastic model has to satisfy is the positivity of its solutions' components. This {constraint can be seen as the counterpart of quadratic-energy conservation in the present context; it} leads to the following  natural modification of an MSM in its EMR formulation \eqref{Eq_EMR} 
for such inverse models, namely   
\be\label{Eq_EMR2} 
\begin{aligned} 
& \hspace*{1.4em}x_{k+1}=\Pi_{\mathcal{K}_{\epsilon}}\Big(x_k+\big[-A_Nx_k +B_N(x_k,x_k)+F_N\big]\delta t +  r_k^{(0)}\delta t\Big), \quad 1\leq k\leq l,\\
& \hspace*{1.4em}  r_{k+1}^{(m-1)}- r_{k}^{(m-1)}= L^{(m)}_N\big[(x_k)^T,( r_k^{(0)})^T,...,( r_k^{(m-1)})^T\big]^T\delta  t+ r_k^{(m)}\delta t, \quad 1\leq m \leq p.
\end{aligned} 
\ee
We adopt here the notations of Sec.~\ref{sec:emr_form}, where 
$x_k=(N_1(t_k),N_2(t_k),N_3(t_k))$, 
$t_k=k\delta t$, 
the $r^{(m)}_k$ are three-dimensional vectors, and we recall that $p$ is the number of levels for which the stopping criterion of \ref{App1} is met.

Here $\Pi_{\mathcal{K}_{\epsilon}}$ denotes the projection onto the convex set 
\be
\mathcal{K}_{\epsilon}:=\{ (N_1,N_2,N_3)\in \mathbb{R}^3_+\,:\, N_i\geq \epsilon,\; 1\leq i\leq 3\},
\ee
for some appropriately chosen $\epsilon>0$ so that
\be
P\widehat{\mathcal{A}}\subset \mathcal{K}_{\epsilon},
\ee
where $P\widehat{\mathcal{A}}$ denotes the projection of $\widehat{\mathcal{A}}$ onto $\mathbb{R}^3$.
Typically, $\epsilon$ can be chosen to be any positive number smaller than the minimum of the { first three } observed components over the
simulation { interval $0 < t \le l$,} which
led us to choose $\epsilon=0.12$.  

When the last-level residual in \eqref{Eq_EMR2}  is well approximated by an independent and identically distributed (i.i.d.) Gaussian random vector (see Remark \ref{Rem:scaling}), the presence of the projection $\Pi_{\mathcal{K}_{\epsilon}}$ in \eqref{Eq_EMR2} implies that the 
resulting recursive stochastic process belongs to a class of (discrete) {\it reflected diffusion processes}; see e.g.~\cite{dupuis1987large}.  
It is the presence of this projector that ensures the vector $x_k$ 
{will have necessarily positive components when 
it obeys} \eqref{Eq_EMR2}, in which the last residual has been replaced by an i.i.d.~Gaussian noise.

To simplify the estimation procedure of the coefficients in \eqref{Eq_EMR2}, a 
{straightforward} multilevel regression procedure, such as described 
in Sec.~\ref{sec:emr_form}, is performed 
{but} the projection is removed 
{at first.} The simulation step is then performed by including the projection $\Pi_{\mathcal{K}_{\epsilon}}$, according to \eqref{Eq_EMR2},  
 in order to avoid the negative values that could lead to blow-up.

Figure \ref{Fig_ACF_LV} show that 
the autocorrelation functions 
of the observed variables $(N_1, N_2,N_3)$ 
{are very well reproduced} by an MSM \eqref{Eq_EMR2} 
{with} $p=14$ extra layers of hidden variables. The peculiar Shilnikov-like shape of the attractor $\widehat{\mathcal{A}}$ is also captured with relatively good accuracy, { as shown by a comparison of panels (a) and (b) in Fig.~\ref{Fig_Att_LV}.}

\begin{figure}
\vspace{-2ex}\centering
\includegraphics[height=0.5\textwidth, width=.75\textwidth]{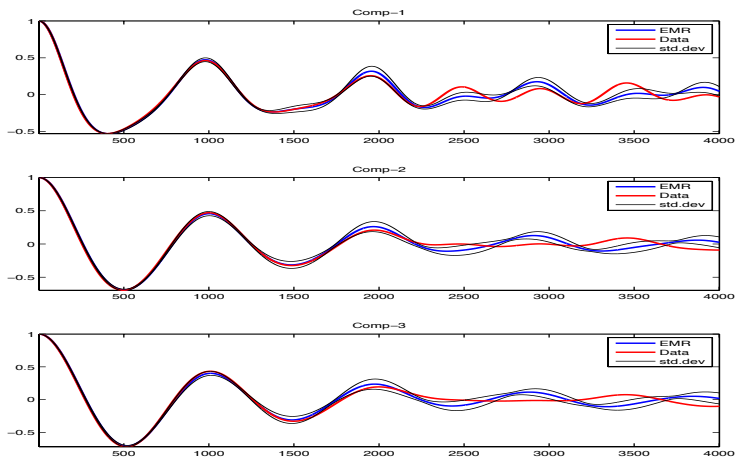}
\vspace{-3ex}
\caption{\small Autocorrelation functions 
of the observed variables $(N_1, N_2,N_3)$ (red curves) obtained by integration of   \eqref{LV_model} with $a_{ij}$ and $b$ given in \eqref{A}, vs.~those estimated from the MSM-simulated  dynamics \eqref{Eq_EMR2} (blue curves), along with their standard deviations (black curves). } 
\label{Fig_ACF_LV}
\end{figure}

\begin{figure}[htbp]
\vspace{-2ex}\centering
\includegraphics[height=0.5\textwidth, width=1\textwidth]{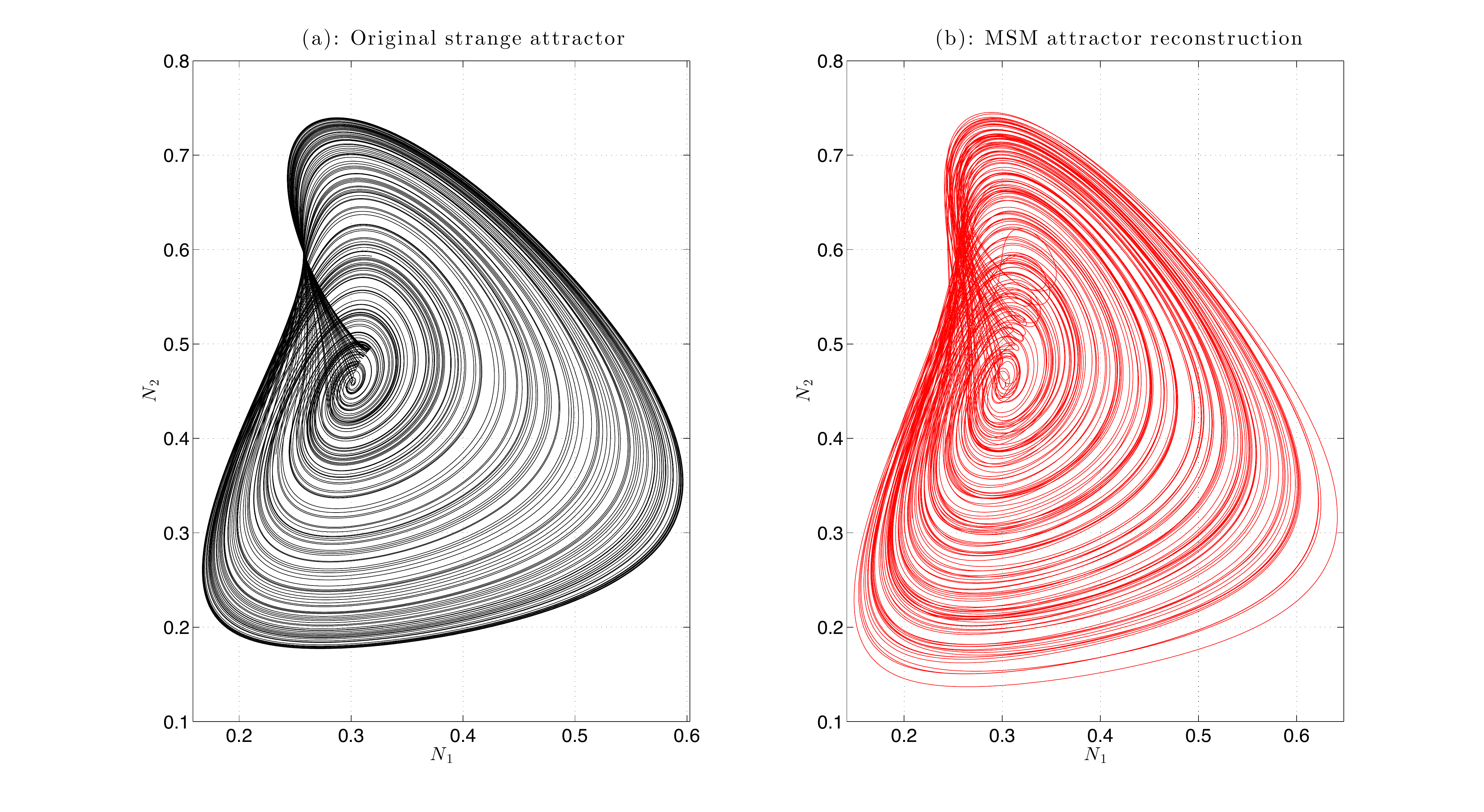}
\vspace{-4ex}
\caption{\small {Strange attractor of the system \eqref{LV_model}: (a)} Original strange attractor of the system; 
and (b) its MSM reconstruction, { both 
projected onto the $(N_1, N_2)$ plane. Panel (b) is obtained by simple forward integration of  \eqref{Eq_EMR2} for an arbitrary realization of the estimated i.i.d. Gaussian noise.}}
\label{Fig_Att_LV}
\end{figure}

The fine, fractal-like structure visible in the MSM attractor reconstruction { of Fig.~\ref{Fig_Att_LV}b} is, as explained below, suggestive of the presence of degenerate noise in the MSM model.   
A closer look at the numerically estimated coefficients in Eq.~\eqref{Eq_EMR2} 
reveals that the spectrum of the grand  linear part of the model
--- which involves the variables $x,r^{(0)},\ldots,r^{(p-1)}$, along with the associated matrices --- contains three unstable modes, while the noise part forces only a low-dimensional subspace, spanned by a few stable, decaying modes. Such a 
combination of unstable modes and partial forcing 
favors a violation of the H\"ormander condition. 
Although we did not formally verify  that this is the case,  
{Fig.~\ref{Fig_Att_LV}b} clearly shows that simple forward integration of Eq.~\eqref{Eq_EMR2} does not reveal  the symptomatic fuzziness that such a plot typically displays when 
the SDE's generator 
{is hypoelliptic;} see \cite{csg11}. The fractal-like features observed here { in Fig.~\ref{Fig_Att_LV}b} recall Remark \ref{Rem:Levy} in Sec.~\ref{Sec_MSM}, to the extent that they argue for
overall stable behavior being possible in the presence of {linearly unstable modes and of} degenerate noise.

The relatively large number of extra layers, { $p = 14$,} can be significantly reduced 
{by taking into account} the particular structure of the right-hand side of \eqref{LV_model} in formulating the EMR system, 
i.e.~by using some {\it a priori} knowledge on the $f_i$, the $g_i$ and the $h$ to be specified in \eqref{Eq_MSMgen}.  This number can be even further reduced if the geometric constraint \eqref{Eq_simplex} is taken into account. 
The purpose {here was merely} to show that, even without 
any {\it a priori} knowledge on the dynamics, the learning of an efficient MSM closure is still possible by the use of simple multilevel linear regression techniques, such as
{those} proposed in \cite{kkg05, kkg05_enso, kkg09_rev} and
described in Sec.~\ref{sec:emr_form} here.


\section{Acknowledgments}
\label{sec:thanks}

We thank the reviewers for their very useful and constructive comments.
The collaboration with several colleagues at Columbia University has helped motivate the data-driven focus of this article. We thank S. Kravstov for useful discussions and two anonymous reviewers for stimulating comments. Preliminary results of this work were presented by DK at the workshop on ``Non-equilibrium Statistical Mechanics and the Theory of Extreme Events in Earth Science,'' held at the Isaac Newton Institute for Mathematical Sciences in November 2013. 

This study was supported by grant  N00014-12-1-0911 from the Multi-University Research Initiative of the Office of Naval Research, by the National Science Foundation grants 
DMS-1049253 and 
OCE-1243175, and by grant DOE-DE-F0A-0000411 from the Department of Energy.

\appendix

\section{EMR stopping criteria} 
\label{App1}
The stopping criterion for adding levels  to an EMR model is based on empirically testing (i) whether the autocorrelations of the last-level residual $ \zeta_t \equiv  r_k^{(p)} \delta t$  approach zero; and (ii) 
{whether} its covariance matrix at zero lag converges to a constant matrix { $\Sigma = \zeta_t^{\rm T}\zeta_t$.} For testing purposes, we can assume without loss of generality that  $\delta t = 1$. Part (i) of this stopping test is rooted in a standard Durbin-Watson statistical test \cite{DurbinWatson}. This test can also be understood in terms of 
{determining} the variance of the regression residual for the additional level that would be added to { the EMR} model {by considering the increments} $\delta \zeta_k:=\zeta_{k+1}-\zeta_{k}$.

When $\zeta_t$ is reasonably well approximated by a random Gaussian variable, i.e. $\zeta_t \sim \ \mathcal{N}(0,Q) $,  performing an extra regression 
for $\delta \zeta_k$, 
according to \eqref{equ:higher_level}, results typically {in} regression coefficients that all approach zero, except the one corresponding to $\zeta_{k}$ itself, 
which approaches $-1$, namely
\be\label{Eq_EMRstop} 
\zeta_{k+1}-\zeta_{k}\approx - \zeta_{k} + \gamma_{k};
\ee
{this observation is} numerically documented, {for instance,} in \cite[Fig.~5]{kkg05}. Thus the regression residual $\gamma_k$ 
{becomes identical} to $\zeta_{k+1}$ and 
{is just a 
lagged} copy of $\zeta_k$, { with lag 1.} The  
coefficient of determination 
$R_i^2$ for the $i$-th component of $\zeta_k$, {with $i = 1, \ldots, d$,} then becomes: 
\be\label{Eq_Rsq} 
R_i^2=1-\frac{\sum\limits_k \gamma_{i,k}^2}{\sum\limits_k (\zeta_{i,k+1}-\zeta_{i,k})^2}\approx1-\frac{\sum\limits_k \zeta_{i,k+1}^2}{\sum\limits_k (\zeta_{i,k+1}^2+\zeta_{i,k}^2)}\approx 1-\frac{\mbox{var}(\zeta_{i,k})}{2\mbox{var}(\zeta_{i,k})}= 0.5. 
\ee
In other words, the fraction of unexplained variance resulting from the regression \eqref{Eq_EMRstop}  approaches 0.5. 

Convergence of the covariance matrix {of $\zeta_{t} \sim r_k^{(p)}$} to a constant matrix $\Sigma=Q^T Q$ can be checked numerically by computing its eigenvalues at each level of Eq.~\eqref{Eq_EMR}. 
{This} convergence typically 
{coincides} with the convergence of $R_i^2$ to 0.5 in Eq.~\eqref{Eq_Rsq}, for each of the components of  $\zeta_{t}$.


\section{Real-time prediction from an MSM, and initialization of   the hidden variables}\label{App3}
This appendix is concerned with the initialization problem for real-time prediction based on a time-discrete version of an MSM. The notations used 
{herein} are those of {of Sec.~\ref{sec:emr_form} and of the previous appendix.}  

To simplify the presentation, we stick to the case of an EMR in its original formulation, namely a discrete system such as  \eqref{Eq_EMR}.
{To integrate \eqref{Eq_EMR} for predictive purposes requires} some 
attention 
to the initialization of the hidden $r^{(m)}$ variables, {$ m = 1, \ldots, p,$  
on the additional EMR} levels. %

For illustration purposes, we consider here the simple case of a three-level model,  
namely the main level and two additional ones, i.e. $p=2$.
This model is assumed to be {\it trained} on the interval { $(0, T^*)$,
i.e.} the model coefficients are obtained from the multivariate time series of data available {for $t = k\delta t$ 
in} this interval. 
{The next point in the time series, 
$T^*+1$,} belongs to the {\it validation} interval where we want to initiate prediction into the future. In other words, we place ourselves in the { real-time case, 
in which} the model coefficients are no longer updated with new observations 
after $T^*$.

{The} right-hand side of \eqref{Eq_EMR} provides a practical way to initialize the hidden variables. {By assuming, as in \ref{App1} and} without loss of generality, that $\delta t=1$, {one gets 
the} following recurrence relations:
\begin{subequations}\label{Eq_r_est}
\begin{align}
& r^{(0)}_{k} = f_1(x_k)-(x_{k+1}-x_k), \label{recur_1} \\
& r^{(1)}_{k} = L^{(0)}[x_k,r^{(0)}_{k}]-(r^{(0)}_{k+1}-r^{(0)}_{k}); \label{recur_2} 
\end{align}
\end{subequations}
here $f_1(x)=-Ax +B(x,x)+F$, and 
the  {last-level} noise $\xi_k$ has been dropped.

If we assume $x_{T^*}$ to be the last available observed data { point,} then 
\eqref{Eq_r_est} {only provides} the hidden variables  $r^{(0)}_{k}$ and   $r^{(1)}_{k}$ 
{up to their corresponding
last} values, namely $r^{(0)}_{T^*-1}$ and $r^{(1)}_{T^*-2}$, respectively.
Indeed, if $x_{k+1}=x_{T^*}$, then 
{Eq.~\eqref{recur_1} only} provides $r^{(0)}_{T^*-1}$, from which { Eq.~\eqref{recur_2} can only help 
determine}  $r^{(1)}_{T^*-2}$, since $r^{(0)}_{T^*}$ is not available. 

On the other hand,  if we assume that a new observation $x_{T^*+1}$ becomes available 
at $T^*+1$ in the validation interval, the value of the $r^{(1)}$-variable can then be 
{calculated} beyond its previously known value, i.e. 
$r^{(1)}_{T^*-1}$ becomes available according to
\beas\label{Eq_back}
& r^{(0)}_{T^*}=f_1(x_{T^*})-(x_{T^*+1}-x_{T^*}),\\
& r^{(1)}_{T^*-1}=L^{(0)}[x_{T^*-1},r^{(0)}_{T^*-1}]-(r^{(0)}_{T^*}-r^{(0)}_{T^*-1}).
\eeas
After such an initialization of the hidden  $r^{(1)}$-variable beyond its {previously 
computed} value, the model prediction  $\widehat{x}_{T^*+2}$ of $x_{T^*+2}$ from the initial data $x_{T^*+1}$ is then obtained by integrating the $r^{(1)}$-variable from the last level to the main one,  according to:  \bea\label{Eq_forward}
     & r^{(1)}_{T^*} = r^{(1)}_{T^*-1} +  L^{(1)}[x_{T^*-1},r^{(0)}_{T^*-1},r^{(1)}_{T^*-1}]+\zeta_{T^*-1}, \\
      & r^{(0)}_{T^*+1} = r^{(0)}_{T^*} +  L^{(0)}[x_{T^*},r^{(0)}_{T^*}]+r^{(1)}_{T^*}, \\
      & \widehat{x}_{T^*+2} = x_{T^*+1} + f_1(x_{T^*+1}) + r^{(0)}_{T^*+1}.
\eea
Note that both  $r^{(1)}_{T^*}$ and  $r^{(0)}_{T^*+1}$  are now randomized  
due to the presence of the random variable $\zeta_{T^*-1}$ on the last level.  Hence the forecast uncertainty is properly accounted for in $x_{T^*+2}$.  

Such an initialization and forecast procedure {based on Eq.~\eqref{Eq_integlog2}} can obviously be carried out for any number of levels of an EMR. More generally, it can be extended to 
an explicit Euler-Maruyama discretization of any stable MSM,  given in its general form and  satisfying the hypotheses of Theorem \ref{Main_thm}. 
The discussion above emphasizes 
that, for such multilevel systems and for real-time prediction purposes, the forecast has to be started { several} time steps back into the past; { the number of these steps must 
equal} the number of hidden layers in  \eqref{Eq_MSMgen}. As an illustration,  Eq.~\eqref{Eq_back} represents {\it backward} initialization of hidden variables in the past by going from the main level to the last one, followed by their {\it forward}  integration from the last level to the main one into the future, as in  \eqref{Eq_forward}. 
%


\section{Energy conservation and its practical EMR aspects}\label{App_energy}

In this appendix, we identify linear relations among the coefficients $B_{ijk}$ so that Eq.~\eqref{Eq_self_energy-preserv} 
be satisfied. Since it is easier to 
introduce  linear constraints in the least-square estimation of the regression coefficients, we look for such constraints that are, in fact,  both necessary and sufficient for Eq.~\eqref{Eq_self_energy-preserv} to hold. These constraints ensure that {\it quadratic} nonlinearities in an EMR formulation will preserve the energy $\| x\|^2$ for the Euclidean norm $\| \cdot \|$ associated with a given basis. In another basis, these conditions 
remain valid, up to a linear change of coordinates.
There are three types of such linear equality constraints for $B$.

First, the coefficient $B_{iii}$ that corresponds to the quadratic term $x_{i}^2$  in the equation for the time evolution of $x_{i}$ is required to be identically zero: 
\begin{equation}
\label{nlcons1}
B_{iii}=0, \quad 
{ i = 1, \ldots, d.}
\end{equation}
In addition, there is a condition { that involves} 
the quadratic 
interactions $x_jx_k$ and $x_k^2$ in the equation for $x_{j}$, and $x_j^2$ and $x_jx_k$ in the equation for $x_{k}$. This condition yields the following skew-symmetric constraints for the two pairs of coefficients:\begin{equation}
\label{nlcons2}
B_{jjk}+B_{kjj}=0, \quad B_{jkk}+B_{kjk}=0, 
{ \quad 1 \le j \neq k \le d.}
\end{equation}
Finally, there are { the quadratic} 
interactions of type $x_jx_k$ in the equation for $x_i$ 
that require the sum of three EMR coefficients to be zero:  
\begin{equation}
\label{nlcons3}
B_{ijk}+B_{jik}+B_{kij}=0, 
{ \quad 1 \le i,  j, k \le d \; } \mbox{ such that}\, (i-j)(j-k)(k-i) \ne 0.
\end{equation}

{A condition like \eqref{Eq_dissipA}} for the matrix $A$ to be positive definite 
can also be included in the regression procedure, as necessary and appropriate.  For particular applications {that are} consistent with geophysical flow models, such as the one considered in Sec.~\ref{sec:Climate_ex},
this condition can be cast as 
a combination of linear equality and inequality constraints. The equality constraint imposes skew-symmetry for off-diagonal terms: 
\begin{equation}
\label{skewcons}
A_{ij}+A_{ji}=0, \quad i\ne j,
\end{equation}
while the { inequality one requires nonnegative} diagonal terms:  
\begin{equation}
\label{poscons}
{ A_{ii} > 0.}
\end{equation}

The total set of constraints is determined by a look-up through all possible occurrences of the above types, Eqs.~\eqref{nlcons1}--\eqref{poscons}, and their number $p$ scales as $p \sim d^2$, where $d$ is the dimension of $x$. {For a quadratic EMR model such as 
\eqref{Eq_EMR}, the total number of model coefficients $P$ is also of order $d^2$. Hence the use of these constraints} reduces considerably the number of independent model coefficients, {from $P$ to $P-p$.} 

Note that energy conserving constraints have been used in model reduction techniques before \cite{Majda_Yuan12,Majda_Harlim2012,MMH}, 
cf.~\citep{Kwasniok96, Kwasniok07}. In both these approaches, however, it is assumed that the full model equations are known and available. In particular, refs.~\cite{Majda_Yuan12,Majda_Harlim2012,MMH} used essentially parameter estimation techniques to obtain the values of the parameters in a model of known structure. 
{The} data-driven methodology proposed in this section, { however,} operates 
{whether} the full governing equations that generated the data are known or not. Hence, Eqs.~\eqref{nlcons1}--\eqref{poscons} are new and of some interest, in particular in the highly realistic and frequently occurring case in which a detailed --- physical, chemical, biological or socio-economic --- model of the process that generated the data is not known.
 
In the unconstrained case, the coefficients for each component $x_i$ of $x$ can be estimated sequentially { by the EMR methodolgy, while} the energy-conserving constraints require estimation of all model coefficients at the main level \textemdash\, {\it i.e.}, $A$, $B$ and $F$  \textemdash\, simultaneously, by using a grand matrix of predictors. The least-square minimization with constraints is solved via quadratic programming with a set of Lagrange multipliers, by using an active-set algorithm or a projection method \cite{Gill_etal1981}.


\section {Interpretation of MSM coefficients}
\label{sec:emd_lim}
Here we focus on the question of interpreting coefficients of EMR {or, more generally,} MSM models that have been constructed in the partial-observing 
situation. The emphasis is on a simple but essential  time-orthogonality property that the observed and hidden variables must satisfy as a consequence of the multilevel regression procedure in the EMR {and} MSM methodology. This  time-orthogonality property typically yields 
different {reduced-model} coefficients than {those of} the original full model, while still 
{allowing the former} to simulate the main statistical features of the observed dynamics. For 
{illustration purposes,} we use  a simple linear-model example borrowed from  \cite{Majda_Yuan12}.

First we point out a basic but important property to be satisfied by the {$r^{(m)}$-variables} 
once the multilevel regression described in Sec.~\ref{sec:emr_form}  has been applied. Let us recall 
that for classical --- {\it i.e.}, unconstrained and nonregularized --- least-square minimization, and in continuous time, a regression residual ${r}^{(0)}(t)$ is necessarily orthogonal in the time domain to the predictor variables $x$ \cite{hastie2009linear}. Likewise, the $r^{(m)}(t)$  variables are estimated 
so as to be orthogonal in the time domain to the variables from the previous levels $({x},{r}^{(0)},\ldots, {r}^{(m-1)})$, 
\beas\label{Eq_MSM_ortho}
& \hspace*{1.4em}\langle x , r^{(m-1)}\rangle_{L^2}=0;\\
& \hspace*{1.4em}\langle r^{(j)}, r^{(m)}\rangle_{L^2}=0, 
\quad 1 \le j, m \le { p}-1 \;\textrm{ with } j \leq m-1,
\eeas
with respect to the $L^2((0,T^*);\mathbb{R}^d)$ inner product defined by:
 \be
 \langle {z}, {y}\rangle_{L^2} := \sum_{i=1}^d \frac{1}{T^*}\int_{0}^{T^*}  z_i (t)y_i(t)\d t,
\ee
and for any vector-valued functions${z}(t) = (z_1(t), \ldots, z_d(t))^{{\mathrm T}}$ and ${y}=(y_1(t), \ldots, y_d(t))^{{\mathrm T}}$.

The simple two-variable, linear model 
from \cite{Majda_Yuan12} {that we use is given by} 
\begin{equation}
\label{EMR_jj}
\left(\begin{array}{c} \d x \\ \d y \end{array}\right) =\left(\begin{array}{cc} a & 1 \\ q & A \end{array}\right)\left(\begin{array}{c} x \\ y \end{array}\right)\d t + \left(\begin{array}{c} 0 \\ \sigma \d W \end{array}\right);
\end{equation}
$x$ here stands for a slow, resolved variable, 
$y$ for a fast, unresolved one,  {while the model coefficients are} $a = -2, q = 1$ and $A = -1$, 
 with $\sigma > 0$. By using 
the standard Euler-Maruyama scheme for 
SDEs \cite[p.~305]{KP92}, one integrates
the following finite-difference version of \eqref{EMR_jj}:
\begin{equation}
\label{EMR_eu}
\left(\begin{array}{c} x_{i+1}- x_i\\ y_{i+1}-y_i \end{array}\right) =\left(\begin{array}{cc} a & 1 \\ 
q & A \end{array}\right)\left(\begin{array}{c} x_i \\ y_i \end{array}\right){\delta}t + \left(\begin{array}{c} 0 \\  \sigma \xi_i \sqrt{\delta t} \end{array}\right),
\end{equation}
 where the $\xi_i$'s are real-valued random variables drawn independently from a normal distribution $\mathcal{N}(0,1)$.

In \cite{Majda_Yuan12}, standard least-square techniques (see appendices therein)
were used to derive analytically an EMR model able to reproduce statistical features  of the dynamics  \textemdash\,  such as decay of correlations or the PDF \textemdash\, of the $x$-variable alone, by using only a finite-length time series $\{x_i : i=1,...,N\}$ of the resolved variable obtained by { integrating the full system given by} Eq. \eqref{EMR_eu}. 
According to Theorem 4.2 in \cite{Majda_Yuan12}, in the limit of $N \rightarrow \infty$ and $\delta t \rightarrow0$, the resulting EMR model's linear part for the { $(x, r)^{\rm T}$}  vector has different coefficients but the same eigenvalues as the original $(x, y)$ model of Eqs.~\eqref{EMR_eu}. 

We stress that the difference in model coefficients 
is solely due to the change of basis imposed by the orthogonal dynamics of the ``hidden'' EMR variable $r$, according to {the conditions given in} \eqref{Eq_MSM_ortho}. 
The related similarity transformation {$S: (x,y)\rightarrow(x,r)$} of the model's linear part  is simply 
\begin{equation}
\label{EMR_trmat}
\left(\begin{array}{c} x \\ y \end{array}\right) = S \left(\begin{array}{c} x \\ r \end{array}\right), \textrm{ with } S= \left(\begin{array}{cc} 1 & 0 \\ -a & 1 \end{array}\right), \textrm{ and }  S^{-1}= \left(\begin{array}{cc} 1 & 0 \\ a & 1 \end{array}\right).
\end{equation}
 {In fact, the}  transformed 
version of the full model in continuous time, Eq.~\eqref{EMR_jj}, 
{is} given by the MSM 
\begin{equation}
\label{EMR_tr}
\left(\begin{array}{c} \d x  \\  \d r  \end{array}\right) =S^{-1}\left(\begin{array}{cc} a & 1 \\ q & A \end{array}\right)S\left(\begin{array}{c} x \\ r \end{array}\right) \d t + \left(\begin{array}{c} 0 \\ \sigma  \d W\end{array}\right), 
\end{equation}
while the transformed version of the full model's discrete form, Eq.~\eqref{EMR_eu}, is equivalent to  the EMR model: 

\begin{equation}
\label{EMR_est}
\left(\begin{array}{c} x_{i+1}- x_i\\ r_{i+1}-r_i \end{array}\right) =\left(\begin{array}{cc} 0 & 1 \\ q-Aa & a+A \end{array}\right)\left(\begin{array}{c} x_i \\ r_i \end{array}\right)
{\delta t} + \left(\begin{array}{c} 0 \\ \sigma \xi_i  \sqrt{{\delta}t} \end{array}\right).
\end{equation}

{Thus,} the linear part of Eq.~\eqref{EMR_est} shares the same eigenvalues as the linear part of  Eq.~\eqref{EMR_eu}, 
{while} the random forcing in Eq.~\eqref{EMR_est} is of the same amplitude, and the same nature, as in Eq.~\eqref{EMR_eu}; hence the model statistics of the simulated variables are identical in the limit of $N \rightarrow \infty$ and $\delta t\rightarrow0$. The only difference between Eq.~\eqref{EMR_est} and Eq.~\eqref{EMR_eu} is that the hidden variable $r$  
is orthogonal to $x$ \textemdash\, while the original $y$ variable is not.

The highly idealized situation analyzed here can obviously be generalized to a broad class of stochastically forced  systems that might, in particular, have a much higher number of  
resolved  as well as hidden variables.

\end{document}